\documentclass[reqno,11pt]{amsart}
\usepackage{amsfonts, amsmath, amssymb, amsthm,   color}
 \usepackage{amsfonts}
\usepackage{amsmath}
\allowdisplaybreaks[1]

\def\bbm[#1]{\mbox{\boldmath $#1$}}

\newtheorem{thm}{Theorem}[section]
\newtheorem{lemma}[thm]{Lemma}
\newtheorem{prop}[thm]{Proposition}

\newtheorem{rmk}[thm]{Remark}

\theoremstyle{definition}
\numberwithin{equation}{section}

\newcommand{\esse}{\mathbb{S}}
  \newcommand{\C}{\mathbb{C}}
\newcommand{\cal}{\mathcal }
  
  \newcommand{\N}{\mathbb{N}}
\newcommand{\D}{\Delta}
\newcommand{\rr}{\mathbb{R}}
\newcommand{\intr}{\int_{\R^2}}
\newcommand{\R}{\mathbb{R}}
\newcommand{\al}{\alpha}
\newcommand{\de}{\delta}

\newcommand{\la}{\lambda}
\newcommand{\into}{\int_{B_1}}

\newcommand{\e}{\varepsilon}
\renewcommand{\(}{\left(}
\renewcommand{\)}{\right)}

\newcommand{\beq}{\begin{equation}}
\newcommand{\eeq}{\end{equation}}

\def\bbm[#1]{\mbox{\boldmath $#1$}}
\def\Re{{\rm Re}\,}
\def\Im{{\rm Im}}

\begin{document}

\title[Second order theorem]{{\Large{Second order  classification \\ for singular Liouville equations \\ \vspace{1mm} with a coefficient function}}}
\author[Teresa D'Aprile]{ Teresa D'Aprile}
\author[Juncheng Wei]{Juncheng Wei}
\author[Lei Zhang]{Lei Zhang} 
\address[Teresa D'Aprile] {Dipartimento di Matematica, Universit\`a di Roma ``Tor
Vergata", via della Ricerca Scientifica 1, 00133 Roma, Italy.}
\email{daprile@mat.uniroma2.it}
\address[Juncheng Wei]{Department of Mathematics \\ Chinese University of Hong Kong \\ Shatin, NT, Hong Kong}\email{wei@math.cuhk.edu.hk}

\address[Lei Zhang]{Department of Mathematics\\
        University of Florida\\
        1400 Stadium Rd\\
        Gainesville FL 32611}
\email{leizhang@ufl.edu}

\thanks{ T. D’Aprile acknowledges the MUR Excellence
Department Project MatMod@TOV awarded to the Department of Mathematics, University of Rome
Tor Vergata, CUP E83C23000330006.  The research of J. Wei is partially supported by RGC grant "New frontier in singular solutions of nonlinear partial differential equations". The research of  Lei Zhang is partially supported by Simons Foundation Grant SFI-MPS-TSM-00013752.}
\begin{abstract}
In this article we 
are concerned with the existence of blow-up solutions
 to the following boundary value problem $$-\Delta v= \la V(x) |x|^2e^v\;\hbox{ in } B_1,\quad	v=0 \;\hbox{ on }\partial B_1,$$  where $B_1$ is the unit ball in $\R^2$ centered at the origin, $V(x)$ is a positive  smooth  potential, and $\la>0$ is a small parameter. 
We find necessary and sufficient conditions on the potential $V$  for the existence 
of  a blow-up sequence of solutions  tending to
infinity near the origin as $\la\to 0^+$. 
In particular, we obtain a second-order classification of the coefficient function 
$V$ for which (simple) blow-up occurs at the origin.
\bigskip


\end{abstract}

\maketitle

{\Large\section{\bf{ Introduction}}}

Singular Liouville-type equations play a prominent role in the study of problems in conformal geometry and mathematical physics.
The well-known  prescribing Gauss curvature equation, the mean field equation, the Chern–Simons self-dual vortices, the statistical mechanics of two-dimensional turbulence,
as well as systems of equations of  Toda type, provide a few representative examples of their applications (\cite{calipuma}, \cite{chaki}, \cite{joliwa}, \cite{ta1}, \cite{ta2}, \cite{tro}, \cite{ya}). 

In this paper  we consider a smooth and bounded domain $\Omega\subset\R^2$  containing the origin,  and we study the following singular Liouville equation 
 \begin{equation}\label{eq00}
  \left\{
      \begin{aligned}&- \D v = \la |x|^{2\alpha}V(x)e^v
      &  \hbox{ in }&  \Omega,\\
    &  \ v=0 &  \hbox{ on }& \partial \Omega.
  \end{aligned}
    \right. \end{equation}
Here $\la$ is a small positive  parameter,  
$V$ is a positive smooth function, and $\alpha>0$. 

The analysis of this class of equations is usually challenging, since the exponential nonlinear term is closely related to the lack of compactness in the variational approach. One important feature of these equations is the blow-up phenomenon: indeed, when a sequence of solutions tends to infinity near a blow-up point, the asymptotic behavior of the solutions near such a point carries important information that is crucial for  some applications.
To this purpose, refined blow-up techniques have been developed to describe the local asymptotic profile of solution sequences around a blow-up point, where bubbling phenomena occur; we refer to  the following works as a partial collection of references: \cite{bachenlita}, \cite{bata}, \cite{breme}, \cite{chenlin00}, \cite{chenlin0}, \cite{ChenLin}, \cite{delkomu}, \cite{espogropi}, \cite{li}, \cite{lisha}, \cite{mawe}, \cite{mal}, \cite{nasu}, \cite{su}, \cite{tar}, \cite{we}, \cite{zha1}.

The   analysis  of the blow-up behaviour at points away from $0$  is well understood (see, for example \cite{breme}, \cite{chenlin0}, \cite{li}, \cite{lisha}, and references therein).  In  \cite{li}, using the method of moving planes to perform a local analysis of problem \eqref{eq00}, Li proved a uniform estimate for a sequence of solutions $v_k$ having a single blow-up point different from $0$. More precisely,  Li showed that the difference between $v_k$ and a suitably scaled sequence of “standard bubbles” is only of order   $O(1)$.
Later, Chen and Lin (\cite{ChenLin}) and  Zhang (\cite{zhang2006}, \cite{zha1}) further refined this profile analysis, improving the uniform estimate.  This type of approximation has played an important role in a number of applications, such as degree counting theorems (\cite{chenlin0}, \cite{ChenLin}), among others.

One of the main difficulties in the study of blow-up solutions to \eqref{eq00} arises when the blow-up point happens to be the origin, i.e. when it coincides with the location of a singularity. 
 It is known from the works of Kuo and Lin (\cite{KuoLin}) and  Bartolucci and Tarantello (\cite{bata}), later refined in \cite{bachenlita}, that if $\al\not\in\N$ (the set of natural numbers), then blow-up solutions satisfy a spherical Harnack inequality around the singular point $0$. Moreover,  by using an improved argument based on the Pohozaev identity,  they established a result similar
to Li’s, concerning
the quantization phenomena for blow-up solutions: more precisely,  the profile of the solutions
differs from global solutions of a Liouville type equation only by a uniformly
bounded term. Naturally, their method also works when  $\alpha=0$ and provides an alternative proof of Li’s result. Subsequently, the paper \cite{zha1} further  improves their result and establishes an expansion of the solutions near the blow-up points with a sharp error estimate.
 
 However, when the strength of the singular source is an integer, i.e.  $\al=N\in\N$,  we say that $0$ is a quantized singular source, and the so called “non-simple blow-up” phenomenon does occur. This means that the bubbling solutions may not satisfy the spherical  Harnack inequality, and multiple local maxima near the singular source may appear. More precisely, let us consider a sequence $\la_k\to 0^+$ and a corresponding sequence $v_k$ of solutions  to the following perturbed problem

\begin{equation}\label{eq000}
  \left\{
      \begin{aligned}&- \D v_k = \la_k |x|^{2N}V(x)e^{v_k}
      &  \hbox{ in }&  \Omega,\\
    &  \ v_k=0 &  \hbox{ on }& \partial \Omega,
  \end{aligned}
    \right. \end{equation}
    where $N\in\N$.
    For a sequence of blow-up solutions $v_k$, it is standard
to assume a uniform bound on the total mass (see \cite{batta}): there
exists $C > 0$, independent of $k$, such that

\beq\label{unifbound}\la_k\int_\Omega |x|^{2N}V(x)e^{v_k}dx\leq C.\eeq
We also suppose that $v_k$ admits the origin as its only blow-up point in $\Omega$; in other words,
\beq\label{eqq000}\max_{\Omega} v_k(x)\to +\infty,  \eeq
and 
 for any compact set  $K\subset\overline{\Omega}\setminus\{0\}$ there exists a constant $C(K)$ (depending on $K$) such that 
\beq\label{eqq0000}\max_{K}v_k\leq C(K).\eeq 

When $N\in\N$, two cases may occur for bubbling phenomena: $v_k$ satisfies the simple blow-up property at $0$, or $v_k$ is not simply blowing up at $0$. More precisely, one of the following two possibilities occurs.

\

\begin{enumerate}
\item[{\bf{I CASE}}] We say that $v_k$ has the \textit{simple blow-up} property if 
\beq\label{simple}\sup_{\overline{\Omega}}\Big(v_k(x)+2(N+1)\log |x|+\log\la_k\Big) \leq C.\eeq
\end{enumerate}
In fact, the simple blow-up property is a necessary and sufficient condition for the validity of the asymptotic behavior of the sequence $v_k$, after scaling, as a “single" bubble around $0$
; we refer to \cite{selfdual}.

\

\begin{itemize}
\item[{\bf{II CASE}}] We say that $v_k$ is \textit{not-simply blowing up} at $0$  if 
\beq\label{nons}\max_{\overline\Omega}\Big(v_k(x)+2(N+1)\log|x| +\log\la_k\Big)\to +\infty\eeq which is equivalent to saying that the spherical Harnack inequality does not hold for $v_k$.\end{itemize}

 In the latter case, 
  it was established by   \cite{bata} and \cite{KuoLin}  that
 if $v_k$ exhibits a non-simple blow-up profile \eqref{nons},  then  there are $N+1$ local maximum points of $v_k$, denoted by  $\beta_{k,0},\ldots, \beta_{k,N}$, and they are asymptotically evenly distributed on $\esse^1$ after scaling according to their magnitude: more precisely, along a subsequence,
$$
\lim_k\frac{\beta_{k,i}}{|\beta_{k,0}|}=e^{{\rm i}(\theta_0+\frac{2\pi i}{N+1})}, \quad i=0,\ldots, N.$$

\

It follows from the analysis in \cite{selfdual} that the simple blow-up property  always holds when $N\not\in \N$. On the other hand, for $N\in\N$, the origin may become a non-simple blow-up point, due to the richer structure of the solution set of the planar “singular” Liouville equation:

\beq\label{limit}
-\Delta U=|x|^{2N}e^U\quad \hbox{in}\;\; \rr^2,\qquad
\int_{\R^2} |x|^{2N}e^{U(x)}dx<+\infty,
\eeq which arises as the so called “limiting” problem, in the sense that the profile describing $v_k$ after blow-up is given by a solution of \eqref{limit}.
In view of the classification result in \cite{Pratar}, we know that, for all $N\in\N$, all solutions of  problem \eqref{limit} are given, in complex notation,  by the three-parameter family of functions \beq\label{bubble} U_{\tau, b}=
\log  \frac{8(N+1)^2\tau}{ (\tau+|x^{N+1}-  b|^2)^2}\quad
\tau>0,\,b\in \C.
\eeq

Note that, in contrast to the case $N\not\in\N$ (where \eqref{bubble} holds only with $b=0$), the solution $U_{\tau, b}$ need not be radial, and may attain its maximum value at a point different from the origin (in fact, at $N+1$ points corresponding to the $N+1$ complex roots of $b$). Thus, when $N\in\N$, the non-simple blow-up phenomenon may occur. 

However, the question on the existence of 
blow-up sequences satisfying  \eqref{eq000}--\eqref{eqq0000} 
 is far from being completely settled, and  
only partial results are currently known. The construction of solutions exhibiting a 
blow-up profile at the origin has been carried out for problem \eqref{eq000}--\eqref{eqq0000} with $V = 1$ in the following cases: for any fixed integer $N\geq 1$ provided that the singular point is shifted in a 
$\la$-dependent way, namely as $|x-p_\la|^{2N}$  (see \cite{delespomu}); or if $\Omega$ is the unit ball and the weight of the source is a positive number $N = N_\la$ approaching an integer $N$ from the right  (see \cite{dawe}). On the other hand, in \cite{dawezhasym}, an analogous result is obtained  for any fixed positive integer $N$ when the potential $V=V_\la$ is symmetric and $\la$-dependent. To our knowledge, the existence of non-simple blow-up phenomenon for \eqref{eq000}--\eqref{eqq0000},  with a fixed potential $V$ and a fixed integer $N$ independent of $\la$, is still an open problem, even in the case of the ball: the only known example is constructed in 
\cite{dawezhaasy} for a special class of potentials.
 In this paper we address this case and we  shall contribute in this direction. 
\

In particular, we are interested in finding 
conditions on the  coefficient function
 $V$
under which there exists a sequence of solutions   blowing up at $0$ (either simple or non simple blow-up). More precisely, we investigate the situation when 
 $N=1$ and  $\Omega$ is the unit ball $B_1$ centered at the origin:
$$B_1:=\{x\in\R^2\,|\, |x|<1\} ,$$
so that problems \eqref{eq000}--\eqref{eqq0000} become
\begin{equation}\label{eq0}
  \left\{
      \begin{aligned}&- \D v_k = \la_k V(x) |x|^{2} e^{v_k}&  \hbox{ in }&  B_1,\\
    &  \ v_k=0 &  \hbox{ on }& \partial B_1,\\ &\lambda_k \int_{B_1}V(x) |x|^{2}e^{v_k} dx \leq C,
  \end{aligned}
    \right. \end{equation}
and we assume that $v_k$  admits the origin as its only blow-up point in $B_1$: \beq\label{eqq0}\max_{B_1}v_k\to +\infty, \quad \sup_{\e\leq |x|\leq 1}v_k<C(\e)\quad \forall \e>0.\eeq
In addition, we impose standard assumptions on $V$:\beq\label{eqq1} V\in C^3(\overline B_1), \quad \min_{\overline B_1}V>0,\eeq
and 
\beq\label{eqq2} V(0)=1,\qquad \nabla V(0)=0, \qquad 0 \hbox{ is a nondegenerate critical point of } V.\eeq 

We point out that the hypothesis on the vanishing of the first derivatives of the potential $V$ at $0$  is a necessary condition for the validity of  a blow-up behaviour at $0$: indeed,   it has been shown in \cite{WeiZhang1}  that  even in the case when 
 blow-up solutions violate the
spherical Harnack inequality 
near a singular source, the first derivatives of the coefficient function must tend to zero. 

This paper is connected to the research developed in  \cite{WeiZhang3}-\cite{WeiZhang4} and \cite{WeiZhang}-\cite{WeiZhang1}-\cite{WeiZhang2}, where various estimates are established for  singular Liouville equations with quantized singularities in a more general setting and for a generic $N\in\N$.  More precisely, in these works  the authors investigate which vanishing theorems the coefficient function 
$V$ must satisfy when  $0$  is the unique blow-up point in a neighborhood of the origin. In particular it has been proved in \cite{WeiZhang4} and \cite{WeiZhang2} that, in non-simple bubbling situations, the following Laplacian vanishing theorem holds.

\begin{thm}[\bf{\cite{WeiZhang4}, \cite{WeiZhang2}}]\label{thweizhang} Assume that 
$V$ satisfies \eqref{eqq1}. Let $v_k$  be a sequence of non-simple blow-up solutions around the origin satisfying \eqref{eq0}--\eqref{eqq0}.  Then,
$$\Delta V(0)=0.
$$
\end{thm}

We also mention that  in \cite{WeiZhang4} higher-order vanishing theorems are obtained. Moreover, in \cite{linawu} the author  proves that the non-simple nature of the blow-up sequence in Theorem \ref{thweizhang} is essential. Indeed, if this assumption is violated, an example can be constructed for which the corresponding Laplacian vanishing property fails. The key point of the proof is the use of a radial coefficient function, which allows one to reduce the problem to the radial case and avoid kernel functions in the linearized equation associated with the quantized setting. 

The main contribution of this paper is to strengthen the conclusions of {\cite{WeiZhang4}, \cite{WeiZhang2}, \cite{linawu}, thereby  yielding a  complete classification 
of the blow-up picture at $0$ for $N=1$. 
Indeed, we establish that  a necessary and sufficient condition on the potential $V$, in addition to the natural assumptions \eqref{eqq1}--\eqref{eqq2}, for the presence of 
bubbling phenomena at $0$ for problem \eqref{eq0}--\eqref{eqq0} is expressed in terms of the second derivatives of $V$ 
 at the origin.
More specifically, using assumptions \eqref{eqq1}–\eqref{eqq2}, and up to a rotation of the coordinates, we may assume that $V$ admits the following expansion in a small neighborhood of the origin: \beq\label{rmmm}V(x)=1+ \frac{\gamma_1 x_1^2+\gamma_2 x_2^2}{2}+O(|x|^3) \hbox{ as }x\to 0.\eeq

\begin{thm}\label{th1} Assume that 
$V$ satisfies \eqref{eqq1}--\eqref{eqq2} and that, up to a rotation, \eqref{rmmm} holds. 
Then, a necessary and sufficient condition for the existence of a  blow-up sequence  $v_k$   satisfying \eqref{eq0}--\eqref{eqq0}
is that \beq\label{pot}\operatorname{det}D^2 V(0)=\gamma_1\cdot\gamma_2>0.
\eeq
 Furthermore, $v_k$ satisfies the simple blow-up property \eqref{simple} and the limiting profile of $v_k$, after a suitable scaling, is a bubble \eqref{bubble} with $N=1$ and  $b=
 (b_1,0) $, where $$b_1>0 \text{ if } |\gamma_1|<|\gamma_2|,\qquad
b_1<0 \text{ if } |\gamma_1|>|\gamma_2|,\qquad
b_1=0 \text{ if } \gamma_1=\gamma_2.
$$ 
\end{thm}

\begin{rmk}\label{th1rem} Theorem \ref{th1} proves that if $0$ is a nondegenerate critical point of $V$, then  not only are  non-simple blow-up situations   ruled out, in agreement with  \cite{WeiZhang4}-\cite{WeiZhang2}, but  simple blow-up can occur if and  only if the  eigenvalues of  the Hessian matrix of $V$ at $0$  have the same sign. Moreover, the corresponding vector $b$ is an eigenvector;
 in particular, if $V$ is radial at the leading order--namely when the eigenvalues coincide--then  the sequence is radially symmetric in its first order approximation and behaves as a  bubble around the maximum point $0$.
\end{rmk}

 The proof of the necessary condition relies on the derivation of a Pohozaev-type  identity, obtained  by suitably choosing
the test functions.  In particular,  we exploit the invariance of the non-radially symmetric solutions for the limiting problem \eqref{limit} with respect to inversion $x\mapsto \frac{x}{|x|^2}$,  
and we employ a test function singular at $0$ to derive 
a Pohozaev identity in the punctured domain $B_1\setminus \{|x|\leq \e\}$;  we then use the  estimates   for bubbling solutions obtained in \cite{bara2} and \cite{WeiZhang}  to
handle the terms arising on the inner boundary $|x|=\e$ in the limit as $\e\to 0^+$. 
A delicate analysis is required to address the boundary terms, for which 
it is crucial  
to consider the case $\Omega=B_1$ in order to  prove that the integral  terms on the outer boundary $\partial B_1$ vanish in the limit $\la\to 0^+$. Moreover, the method fails for $N\geq 2$, since in the expansion of the Pohozaev identity we no longer capture information on the Hessian matrix of $V$ (see Remark \ref{fails} for technical details).

 The main tool we use to construct the blow-up solutions (sufficient condition) is the Lyapunov-Schmidt reduction,
which is widely employed in many planar critical problems. Roughly speaking, 
we look for a solution to \eqref{eq0} in a small neighborhood of the first approximation provided by the bubble \eqref{bubble},  with a suitable choice of the parameters $\tau, b$. As is quite standard in singular perturbation theory, a crucial ingredient is the nondegeneracy of the explicit
limit
family of solutions \eqref{bubble} to the limiting Liouville problem \eqref{limit}, in the sense that all bounded elements in the kernel of the linearized operator correspond to variations along the parameters of
the family, as established in \cite{delespomu}.  By carrying out the reduction process,  the problem of constructing blow-up families of solutions to \eqref{eq0}--\eqref{eqq0} is reduced  to the problem of finding critical points of a functional depending on the parameter $b$. More specifically, the reduced functional turns out to be
 uniformly close to a limiting  functional ${\mathcal F}$, which admits  $0$ as   a nondegenerate critical point;  consequently, the persistence of a critical point    is guaranteed for sufficiently small uniform perturbations. This allows us to detect a family of bubbling solutions to \eqref{eq0}--\eqref{eqq0}.

 \
 
 The organization of the article is as follows. In section 2, we prove the   necessary condition for the validity of a blow-up behaviour at $0$. We first   derive two Pohozaev identities for  problem \eqref{eq0}--\eqref{eqq0}; next, using as a crucial ingredient the profile analysis for sequences of solutions having a single blow-up point at $0$  obtained in \cite{bara2} and \cite{WeiZhang}, 
we exploit the Pohozaev identities in a delicate way and derive an integral estimate involving the Hessian matrix of $V$. 

In section 3, we present a construction of bubbling solutions based on the Lyapunov-Schmidt reduction method. We first estimate the error up to which the approximating solution solves problem \eqref{eq0}; next, we prove
the invertibility of the linearized operator, and finally we address  the solvability of the  reduced
problem by a contraction argument. This  completes the proof of Theorems \ref{th1}.

\

\noindent{\bf{NOTATION}}: Throughout the paper, we will frequently denote by $C>0$, $c>0$ fixed
constants that may change from line to line, but are always
independent of the variables under consideration. We also use the standard
notations $O(1)$, $o(1)$, $O(\lambda)$, $o(\lambda)$ to describe
the asymptotic behaviors of quantities.

{\Large{\section{\bf{Proof of Theorem \ref{th1} Part I: Necessary Condition for Blow-Up}}}}
In this section, we derive Pohozaev-type identities for solutions of the singular Liouville equation
\beq\label{eq1}\left\{\begin{aligned}& -\Delta v=\la |x|^{2} V(x) e^v&\hbox{ in } &B_1\\ &v=0& \hbox{ on } &\partial B_1\end{aligned}\right..\eeq
  We then exploit these
integral identities to provide the necessary condition on the potential $V$ stated in Theorem \ref{th1} for the existence of associated families of 
blow-up solutions. 

\

\subsection{{\large{Pohozaev-type identities.}}} Throughout this section, we assume that 
the potential $V>0$ belongs to the class $C^1(\overline\Omega)$. 
We then develop some basic properties  of solutions  $v\in C(\overline B_1)$ to the boundary value problem \eqref{eq1}

The following proposition establishes an integral representation for the first derivatives.

\begin{prop}\label{propo1}
Any solution $v$ of \eqref{eq1} satisfies the following Pohozaev identity:
$$\begin{aligned}&\frac12\int_{\partial B_1}\bigg|\frac{\partial v}{\partial \nu}\bigg|^2  x_i d\sigma+\la\int_{\partial B_1} V(x) x_i d\sigma
+2 \int_{\partial B_1} \frac{\partial v}{\partial \nu} x_i d\sigma =4\pi \frac{\partial v}{\partial x_i}(0) +\la\int_{B_1} |x|^{2} \frac{\partial V}{\partial x_i}(x) e^vdx\end{aligned}$$
for $i=1,2,$ where $\nu=(\nu_1,\nu_2)$ stands for the unit outward normal.
\end{prop}

\begin{proof}
Let us multiply both sides of equation in \eqref{eq1} by $\frac{x_1}{|x|^{2}}$ and integrate on $B_1\setminus B_\e$ for $\e>0$ sufficiently small:
\beq\label{ple0}\int_{B_1\setminus B_\e}(-\Delta v)\frac{x_1}{|x|^{2}}dx=\la\int_{B_1\setminus B_\e}V(x)e^vx_1 dx .\eeq
Taking into account that $\frac{x_i}{|x|^{2}}$ is a harmonic function for $x\neq 0$, by applying Gauss-Green formula we derive
\beq\label{ple}\int_{B_1\setminus B_\e}(-\Delta v)\frac{x_1}{|x|^{2}}dx= -\int_{\partial B_1\cup \partial B_\e}\frac{\partial v}{\partial \nu}\frac{x_1}{|x|^{2}}d\sigma+ \int_{\partial B_\e}\frac{\partial }{\partial \nu}\bigg(\frac{x_1}{|x|^{2}}\bigg) v\,d\sigma.\eeq
Let us examine separately the boundary integrals over $\partial B_\e$: first
\beq\label{ple1}-\int_{ \partial B_\e}\frac{\partial v}{\partial \nu}\frac{x_1}{|x|^{2}}d\sigma=- \frac{1}{\e^{2}}\int_{ \partial B_\e}\frac{\partial v}{\partial \nu}x_1d\sigma= \frac{1}{\e^{3}}\int_{ \partial B_\e} \bigg(\frac{\partial v}{\partial x_1} x_1+\frac{\partial v}{\partial x_2} x_2\bigg)x_1d\sigma.
\eeq  Since $v$ is two times differentiable  by standard regularity theory, 
we get$$\frac{\partial v}{\partial x_1} x_1+\frac{\partial v}{\partial x_2} x_2=\frac{\partial v}{\partial x_1} (0)x_1+ \frac{\partial v}{\partial x_2} (0)x_2+O(|x|^2).$$
Using that  $\int_{\partial B_\e} x_1x_2d\sigma=0$ by oddness  and $\int_{\partial B_\e} x_i^2d\sigma=\pi\e^{3}$, 
by \eqref{ple1} we deduce  
\beq\label{ple2}-\int_{ \partial B_\e}\frac{\partial v}{\partial \nu}\frac{x_1}{|x|^{2}}d\sigma=\pi\frac{\partial v}{\partial x_1} (0)+o(1)\eeq as $\e\to 0^+$.
Let us pass to examine the second boundary term in \eqref{ple}: 
\beq\label{ple4}\begin{aligned}\int_{\partial B_\e}\frac{\partial }{\partial \nu}\bigg(\frac{x_1}{|x|^{2}}\bigg) v\,d\sigma&=-\frac1\e\int_{\partial B_\e}\bigg(\frac{\partial }{\partial x_1}\bigg(\frac{x_1}{|x|^{2}}\bigg) x_1+ \frac{\partial }{\partial x_2}\bigg(\frac{x_1}{|x|^{2}}\bigg) x_2\bigg)v(x)\,d\sigma
\\ &= \frac{1}{\e^{3}}\int_{\partial B_\e}x_1v(x)d\sigma. 
\end{aligned}\eeq
Since $$x_1v(x)=x_1v(0)+\frac{\partial v}{\partial x_1} (0)x_1^2+ \frac{\partial v}{\partial x_2} (0)x_1x_2+O(|x|^3),$$ by \eqref{ple4} we derive \beq\label{ple5}\begin{aligned}\int_{\partial B_\e}\frac{\partial }{\partial \nu}\bigg(\frac{x_1}{|x|^{2}}\bigg) v\,d\sigma&= \pi\frac{\partial v}{\partial x_1} (0) +o(1)
\end{aligned}\eeq
as $\e\to 0^+$. 
By inserting \eqref{ple2} and \eqref{ple5} into \eqref{ple} we obtain
$$\int_{B_1\setminus B_\e}(-\Delta v)\frac{x_1}{|x|^{2}}dx= 2\pi\frac{\partial v}{\partial x_1} (0)-\int_{\partial B_1}\frac{\partial v}{\partial \nu} x_1d\sigma+o(1)$$
as $\e\to 0^+$ by which, letting $\e$ go to $0^+$ in \eqref{ple0}, we conclude
\beq\label{ple6}\la\int_{B_1}V(x)e^v x_1 dx=2\pi \frac{\partial v}{\partial x_1} (0) -\int_{\partial B_1}\frac{\partial v}{\partial \nu}x_1d\sigma.\eeq

Next, let us multiply both sides of the equation in \eqref{eq1} by $\frac{\partial v}{\partial x_1}$   and integrating on $B_1$ we get
\beq\label{nirv}\int_{B_1}  \nabla v\cdot \nabla \bigg(\frac{\partial v}{\partial x_1}\bigg) dx-\int_{\partial B_1}\frac{\partial v}{\partial \nu}\frac{\partial v}{\partial x_1} d\sigma=\int_{B_1}\la |x|^{2} V(x) e^v\frac{\partial v}{\partial x_1} dx.\eeq
Taking into account that $$\nabla v\cdot \nabla \bigg(\frac{\partial v}{\partial x_1}\bigg) =\frac12\frac{\partial }{\partial x_1}\Big(|\nabla v|^2\Big), \qquad e^v\frac{\partial v}{\partial x_1} =\frac{\partial e^v}{\partial x_1} ,$$ 
after integration by parts  \eqref{nirv} becomes
$$\begin{aligned}&\frac12\int_{\partial B_1}|\nabla v|^2 x_1 d\sigma-\int_{\partial B_1}\frac{\partial v}{\partial \nu}\frac{\partial v}{\partial x_1} d\sigma
\\ &=-2\la\int_{B_1} x_1 V(x) e^vdx-\la\int_{B_1} |x|^{2} \frac{\partial V}{\partial x_1}(x) e^vdx+\la\int_{\partial B_1} V(x) x_i d\sigma\end{aligned}$$
where we have used the homogeneous boundary condition $v=0$ on $\partial B_1$. Since  on $\partial B_1$ one has $\nabla v=\frac{\partial v}{\partial \nu}\,\nu=\frac{\partial v}{\partial \nu}\,x$,   the above identity can be rewritten as 
$$\begin{aligned}&-\frac12\int_{\partial B_1}\bigg|\frac{\partial v}{\partial \nu}\bigg|^2x_1 d\sigma
=-2\la\int_{B_1} x_1 V(x) e^vdx-\la\int_{B_1} |x|^{2} \frac{\partial V}{\partial x_1}(x) e^vdx+\la\int_{\partial B_1} V(x) x_1 d\sigma\end{aligned}$$

Combining the last identity with \eqref{ple6}
 we get the thesis for $i=1$. The identity for $i=2$ can be deduced similarly.
\end{proof}

We are going to provide  two crucial Pohozaev identities satisfied by   solutions of \eqref{eq1}.

\begin{prop}\label{propo2} Any solution  $v$ of \eqref{eq1} satisfies the following two Pohozaev identities:
$$\begin{aligned}\int_{B_1} \la e^{v}&\bigg(\frac{\partial V}{\partial x_1}x_1-\frac{\partial V}{\partial x_2}x_2\bigg) dx \\ &=\la\int_{\partial B_1}  V(x)(x_1^2-x_2^2)d\sigma+\frac12\int_{\partial B_1}\bigg|\frac{\partial v}{\partial \nu}\bigg|^2(x_1^2-x_2^2) d\sigma
-\pi\bigg|\frac{\partial v}{\partial x_1}(0)\bigg|^2+\pi\bigg|\frac{\partial v}{\partial x_2}(0)\bigg|^2,
\end{aligned}$$

$$\begin{aligned}\int_{B_1} \la e^{v}\bigg(\frac{\partial V}{\partial x_1}&x_2+\frac{\partial V}{\partial x_2}x_1\bigg) dx =2\la\int_{\partial B_1}  V(x)x_1x_2d\sigma+\int_{\partial B_1}\bigg|\frac{\partial v}{\partial \nu}\bigg|^2 x_1x_2 d\sigma
-2\pi
\frac{\partial v}{\partial x_1}(0)\frac{\partial v}{\partial x_2}(0),
\end{aligned}$$
where $\nu=(\nu_1,\nu_2)$ stands for the unit outward normal. \end{prop}

\begin{proof} Let us multiply  both sides of the equation in \eqref{eq1} by $$\frac{\partial v}{\partial x_1}\frac{x_1}{|x|^{2}}-\frac{\partial v}{\partial x_2}\frac{x_2}{|x|^{2}};$$ using that $v$ belongs to $C^2(\overline{B}_1)$ by standard regularity theory   and integrating on $B_1\setminus B_\e$ for $\e>0$ sufficiently small we get 
\beq\label{leftright}\begin{aligned}\int_{B_1\setminus B_\e} (-\Delta v)\bigg(\frac{\partial v}{\partial x_1}\frac{x_1}{|x|^{2}}&-\frac{\partial v}{\partial x_2}\frac{x_2}{|x|^{2}}\bigg)dx=\la \int_{B_1\setminus B_\e}V(x) e^{v}\bigg(\frac{\partial v}{\partial x_1}x_1-\frac{\partial v}{\partial x_2}x_2\bigg)dx.\end{aligned}\eeq
By applying Gauss Green formula, we have
$$\begin{aligned}&\int_{B_1\setminus B_\e} (-\Delta v)\bigg(\frac{\partial v}{\partial x_1}\frac{x_1}{|x|^{2}}-\frac{\partial v}{\partial x_2}\frac{x_2}{|x|^{2}}\bigg)dx\\ &=
\int_{B_1\setminus B_\e} \nabla v\cdot \nabla\bigg(\frac{\partial v}{\partial x_1}\frac{x_1}{|x|^{2}}-\frac{\partial v}{\partial x_2}\frac{x_2}{|x|^{2}}\bigg)dx-\int_{\partial B_1\cup\partial B_\e }\frac{\partial v}{\partial \nu}\bigg(\frac{\partial v}{\partial x_1}\frac{x_1}{|x|^{2}}-\frac{\partial v}{\partial x_2}\frac{x_2}{|x|^{2}}\bigg)d\sigma .
\end{aligned} $$
By direct computations  we derive
$$\begin{aligned}\nabla v\cdot \nabla \bigg(\frac{\partial v}{\partial x_1}\frac{x_1}{|x|^2}-\frac{\partial v}{\partial x_2}\frac{x_2}{|x|^2}\bigg) &=
\frac12\frac{\partial }{\partial x_1}(|\nabla v|^2)\frac{x_1}{|x|^2}-\frac12\frac{\partial }{\partial x_2}(|\nabla v|^2)\frac{x_2}{|x|^2}-|\nabla v|^2 \frac{x_1^{2}-x_2^2}{|x|^{4}}
\\ &=\frac12\frac{\partial}{\partial x_1}\bigg(|\nabla v|^2 \frac{x_1}{|x|^2}\bigg)-\frac12 \frac{\partial}{\partial x_2}\bigg(|\nabla v|^2 \frac{x_2}{|x|^2}\bigg),
\end{aligned}$$
by which, applying again Gauss-Green formula,
$$\begin{aligned}&\int_{B_1\setminus B_\e} \nabla v\cdot \nabla \bigg(\frac{\partial v}{\partial x_1}\frac{x_1}{|x|^2}-\frac{\partial v}{\partial x_2}\frac{x_2}{|x|^2}\bigg)
 = \int_{\partial B_1\cup\partial B_\e} \frac{|\nabla v|^2}{2} \bigg(\frac{x_1}{|x|^2}\nu_1-\frac{x_2}{|x|^2}\nu_2\bigg) d\sigma
.
\end{aligned}$$
Hence we deduce

\beq\label{csa}\begin{aligned}\int_{B_1\setminus B_\e} (-\Delta v)\bigg(\frac{\partial v}{\partial x_1}\frac{x_1}{|x|^2}-\frac{\partial v}{\partial x_2}\frac{x_2}{|x|^2}\bigg)dx&=
 \int_{\partial B_1\cup\partial B_\e}\frac{ |\nabla v|^2}{2} \bigg(\frac{x_1}{|x|^2}\nu_1-\frac{x_2}{|x|^2}\nu_2\bigg) d\sigma
\\ &\;\;\;\;-\int_{\partial B_1\cup\partial B_\e }\frac{\partial v}{\partial \nu}\bigg(\frac{\partial v}{\partial x_1}\frac{x_1}{|x|^2}-\frac{\partial v}{\partial x_2}\frac{x_2}{|x|^2}\bigg)d\sigma .
\end{aligned}\eeq
Let us examine separately the boundary integrals over $\partial B_1$ and over $\partial B_\e$:
taking into account of the homogeneous boundary condition in problem \eqref{eq1} we have that $\nabla v=\frac{\partial v}{\partial \nu}\,\nu=\frac{\partial v}{\partial \nu}\,x $ on $\partial B_1$, which implies $$\begin{aligned}& \int_{\partial B_1} \frac{|\nabla v|^2 }{2}\bigg(\frac{x_1}{|x|^2}\nu_1-\frac{x_2}{|x|^2}\nu_2\bigg) d\sigma-\int_{\partial B_1}\frac{\partial v}{\partial \nu}\bigg(\frac{\partial v}{\partial x_1}\frac{x_1}{|x|^2}-\frac{\partial v}{\partial x_2}\frac{x_2}{|x|^2}\bigg)d\sigma
=\frac12\int_{\partial B_1}\bigg|\frac{\partial v}{\partial \nu}\bigg|^2 (x_2^2-x_1^2)d\sigma.
\end{aligned}$$
On the other hand,  $\nu=-\frac{x}{\e}$ on $\partial B_\e$; consequently
 $$\begin{aligned}&\int_{\partial B_\e}\frac{ |\nabla v|^2}{2} \bigg(\frac{x_1}{|x|^2}\nu_1-\frac{x_2}{|x|^2}\nu_2\bigg) d\sigma-\int_{\partial B_\e}\frac{\partial v}{\partial \nu}\bigg(\frac{\partial v}{\partial x_1}\frac{x_1}{|x|^2}-\frac{\partial v}{\partial x_2}\frac{x_2}{|x|^2}\bigg)d\sigma
 \\ &=-\frac{1}{\e^{3}} \int_{\partial B_\e} \frac{|\nabla v|^2}{2} \Re(x^2)d\sigma+\frac{1}{\e^{3}} \int_{\partial B_\e} \bigg(\frac{\partial v}{\partial x_1}x_1+\frac{\partial v}{\partial x_2}x_2\bigg)\bigg(\frac{\partial v}{\partial x_1}x_1-\frac{\partial v}{\partial x_2}x_2\bigg)d\sigma
 \\ &
=\frac{1}{2\e}\int_{\partial B_\e}\bigg(\bigg|\frac{\partial v}{\partial x_1}\bigg|^2-\bigg|\frac{\partial v}{\partial x_2}\bigg|^2\bigg)d\sigma
\\ &= \pi \bigg|\frac{\partial v}{\partial x_1}(0)\bigg|^2-\pi\bigg|\frac{\partial v}{\partial x_2}(0)\bigg|^2+o(1)  \end{aligned}$$
as $\e\to 0^+$. 
By inserting the above two boundary estimates into \eqref{csa} we get
\beq\label{csa1}\begin{aligned}\int_{B_1\setminus B_\e} (-\Delta v)\bigg(\frac{\partial v}{\partial x_1}&\frac{x_1}{|x|^2}-\frac{\partial v}{\partial x_2}\frac{x_2}{|x|^2}\bigg)dx\\&=
\frac12 \int_{\partial B_1} \bigg|\frac{\partial  v}{\partial \nu}\bigg|^2 (x_2^2-x_1^2) d\sigma 
+\pi \bigg|\frac{\partial v}{\partial x_1}(0)\bigg|^2-\pi\bigg|\frac{\partial v}{\partial x_2}(0)\bigg|^2+o(1) 
\end{aligned}\eeq as $\e\to 0^+$.

Now let us pass to examine the right hand side of \eqref{leftright}: by integrating by parts we derive
\beq\label{sgre}\begin{aligned}&\int_{B_1} \la  V(x) e^{v}\bigg(\frac{\partial v}{\partial x_1}x_1-\frac{\partial v}{\partial x_2}x_2\bigg)  dx
 \\&=\int_{B_1} \la V(x) \bigg(\frac{\partial e^v}{\partial x_1}x_1-\frac{\partial e^v}{\partial x_2}x_2\bigg)dx\\ &=
-\int_{B_1} \la e^{v}\bigg(\frac{\partial V(x)}{\partial x_1}x_1-\frac{\partial V(x)}{\partial x_2}x_2\bigg) dx
+\int_{\partial B_1}\la  V(x)e^v \Big(x_1\nu_1-x_2\nu_2\Big)d\sigma
\\ &=
-\int_{B_1} \la e^{v}\bigg(\frac{\partial V(x)}{\partial x_1}x_1-\frac{\partial V(x)}{\partial x_2}x_2\bigg) dx
+\int_{\partial B_1}\la  V(x)(x_1^2-x_2^2)d\sigma ,\end{aligned}\eeq  where we have used the homogeneous boundary condition $v=0$ on $\partial B_1$. 
Inserting \eqref{csa1} and \eqref{sgre} into \eqref{leftright} and letting $\e\to 0$,  we obtain the  first of the two Pohozaev identities. 
\

Similarly let us multiply  both sides of the equation in \eqref{eq1} by $$\frac{\partial v}{\partial x_1}\frac{x_2}{|x|^2}+\frac{\partial v}{\partial x_2}\frac{x_1}{|x|^2};$$ integrating on $B_1\setminus B_\e$ for $\e>0$ sufficiently small we get 
\beq\label{leftrightx}\begin{aligned}\int_{B_1\setminus B_\e} (-\Delta v)\bigg(\frac{\partial v}{\partial x_1}\frac{x_2}{|x|^2}&+\frac{\partial v}{\partial x_2}\frac{x_1}{|x|^2}\bigg)dx=\la \int_{B_1\setminus B_\e}V(x) e^{v}\bigg(\frac{\partial v}{\partial x_1}x_2+\frac{\partial v}{\partial x_2}x_1\bigg)dx.\end{aligned}\eeq
Gauss-Green formula gives 
$$\begin{aligned}&\int_{B_1\setminus B_\e} (-\Delta v)\bigg(\frac{\partial v}{\partial x_1}\frac{x_2}{|x|^2}+\frac{\partial v}{\partial x_2}\frac{x_1}{|x|^2}\bigg)dx\\ &=
\int_{B_1\setminus B_\e} \nabla v\cdot \nabla\bigg(\frac{\partial v}{\partial x_1}\frac{x_2}{|x|^2}+\frac{\partial v}{\partial x_2}\frac{x_1}{|x|^2}\bigg)dx-\int_{\partial B_1\cup\partial B_\e }\frac{\partial v}{\partial \nu}\bigg(\frac{\partial v}{\partial x_1}\frac{x_2}{|x|^2}+\frac{\partial v}{\partial x_2}\frac{x_1}{|x|^2}\bigg)d\sigma .
\end{aligned} $$
We compute
$$\begin{aligned}\nabla v\cdot \nabla \bigg(\frac{\partial v}{\partial x_1}\frac{x_2}{|x|^2}+\frac{\partial v}{\partial x_2}\frac{x_1}{|x|^2}\bigg) &=
\frac12\frac{\partial }{\partial x_1}(|\nabla v|^2)\frac{x_2}{|x|^2}+\frac12\frac{\partial }{\partial x_2}(|\nabla v|^2)\frac{x_1}{|x|^2}-|\nabla v|^2 \frac{2x_1x_2}{|x|^{4}}
\\ &=\frac12\frac{\partial}{\partial x_1}\bigg(|\nabla v|^2 \frac{x_2}{|x|^2}\bigg)+\frac12 \frac{\partial}{\partial x_2}\bigg(|\nabla v|^2 \frac{x_1}{|x|^2}\bigg),
\end{aligned}$$
which implies, by using  again Gauss-Green formula,
$$\begin{aligned}&\int_{B_1\setminus B_\e} \nabla v\cdot \nabla \bigg(\frac{\partial v}{\partial x_1}\frac{x_2}{|x|^2}+\frac{\partial v}{\partial x_2}\frac{x_1}{|x|^2}\bigg)
 = \int_{\partial B_1\cup\partial B_\e} \frac{|\nabla v|^2 }{2}\bigg(\frac{x_2}{|x|^2}\nu_1+\frac{x_1}{|x|^2}\nu_2\bigg) d\sigma
.
\end{aligned}$$
Therefore we rewrite the left hand side of \eqref{leftrightx} as

\beq\label{csax}\begin{aligned}\int_{B_1\setminus B_\e} (-\Delta v)\bigg(\frac{\partial v}{\partial x_1}\frac{x_2}{|x|^2}+\frac{\partial v}{\partial x_2}\frac{x_1}{|x|^2}\bigg)dx&=
 \int_{\partial B_1\cup\partial B_\e} \frac{|\nabla v|^2}{2} \bigg(\frac{x_2}{|x|^2}\nu_1+\frac{x_1}{|x|^2}\nu_2\bigg) d\sigma
\\ &\;\;\;\;-\int_{\partial B_1\cup\partial B_\e }\frac{\partial v}{\partial \nu}\bigg(\frac{\partial v}{\partial x_1}\frac{x_2}{|x|^2}+\frac{\partial v}{\partial x_2}\frac{x_1}{|x|^2}\bigg)d\sigma .
\end{aligned}\eeq
Let us focus on  the boundary terms:
the homogeneous boundary condition in \eqref{eq1} yields  $\nabla v=\frac{\partial v}{\partial \nu}\,\nu=\frac{\partial v}{\partial \nu}\,x $ on $\partial B_1$; accordingly $$\begin{aligned}& \int_{\partial B_1} \frac{|\nabla v|^2}{2} \bigg(\frac{x_2}{|x|^2}\nu_1+\frac{x_1}{|x|^2}\nu_2\bigg) d\sigma-\int_{\partial B_1}\frac{\partial v}{\partial \nu}\bigg(\frac{\partial v}{\partial x_1}\frac{x_2}{|x|^2}+\frac{\partial v}{\partial x_2}\frac{x_1}{|x|^2}\bigg)d\sigma
=-\int_{\partial B_1}\bigg|\frac{\partial v}{\partial \nu}\bigg|^2 x_1x_2d\sigma.
\end{aligned}$$
Since $\nu=-\frac{x}{|x|}$ on $\partial B_\e$, we employ the integrals over $\partial B_\e$ as follows 
 $$\begin{aligned}& \int_{\partial B_\e} \frac{|\nabla v|^2}{2} \bigg(\frac{x_2}{|x|^2}\nu_1+\frac{x_1}{|x|^2}\nu_2\bigg) d\sigma-\int_{\partial B_\e}\frac{\partial v}{\partial \nu}\bigg(\frac{\partial v}{\partial x_1}\frac{x_2}{|x|^2}+\frac{\partial v}{\partial x_2}\frac{x_1}{|x|^2}\bigg)d\sigma
 \\ &=-\frac{1}{\e^{3}} \int_{\partial B_\e} |\nabla v|^2 x_1x_2 +\frac{1}{\e^{3}} \int_{\partial B_\e} \bigg(\frac{\partial v}{\partial x_1}x_1+\frac{\partial v}{\partial x_2}x_2\bigg)\bigg(\frac{\partial v}{\partial x_1}x_2+\frac{\partial v}{\partial x_2}x_1\bigg)d\sigma
 \\ &
=
\frac{1}{\e}\int_{\partial B_\e}\frac{\partial v}{\partial x_1}\frac{\partial v}{\partial x_2}d\sigma
\\ &= 2\pi\frac{\partial v}{\partial x_1}(0)\frac{\partial v}{\partial x_2}(0)+o(1)
  \end{aligned}$$ as $\e\to 0^+$.
 By combining the above two boundary estimates with \eqref{csax} we get
\beq\label{csa1x}\begin{aligned}\int_{B_1\setminus B_\e} (-\Delta v)\bigg(&\frac{\partial v}{\partial x_1}\frac{x_2}{|x|^2}+\frac{\partial v}{\partial x_2}\frac{x_1}{|x|^2}\bigg)dx\\ &=
-\int_{\partial B_1} \bigg|\frac{\partial  v}{\partial \nu}\bigg|^2 x_1x_2 d\sigma+ 2\pi\frac{\partial v}{\partial x_1}(0)\frac{\partial v}{\partial x_2}(0)+o(1)\\ &\;\;\;\;
\end{aligned}\eeq as $\e\to 0^+$.

Next we estimate the right hand side of \eqref{leftrightx} similarly as in \eqref{sgre}: \beq\label{sgrex}\begin{aligned}&\int_{B_1} \la  V(x) e^{v}\bigg(\frac{\partial v}{\partial x_1}x_2+\frac{\partial v}{\partial x_2}x_1\bigg)  dx
\\&=\int_{B_1} \la V(x) \bigg(\frac{\partial e^v}{\partial x_1}x_2+\frac{\partial e^v}{\partial x_2}x_1\bigg)dx\\ &=
-\int_{B_1} \la e^{v}\bigg(\frac{\partial V(x)}{\partial x_1}x_2+\frac{\partial V(x)}{\partial x_2}x_1\bigg) dx
+\int_{\partial B_1}\la  V(x)e^v \Big(x_2\nu_1+x_1\nu_2\Big)d\sigma
\\ &=
-\int_{B_1} \la e^{v}\bigg(\frac{\partial V(x)}{\partial x_1}x_2+\frac{\partial V(x)}{\partial x_2}x_1\bigg) dx
+2\int_{\partial B_1}\la  V(x) x_1x_2d\sigma.
\end{aligned}\eeq 
Inserting \eqref{csa1x} and \eqref{sgrex} into \eqref{leftrightx} and letting $\e\to 0$, we find the second Pohozaev identity, which concludes the proof. \end{proof}

\

\subsection{{\large{Uniform behaviour of blowing up solutions}}}
In this section, we derive the asymptotic behaviour of blow-up solutions $v_k$ for the problem \eqref{eq0}--\eqref{eqq0}. This result  is a direct consequence of a uniform estimate established in \cite{bara2}  for bubbling solutions to the singular Liouville equation without boundary conditions and standard elliptic estimates theory.
Let $G$ denote the Green's function which  takes for the unit ball the form $$G(x,y)=\frac{1}{2\pi}\log\frac{1}{|x-y|}+\frac{1}{4\pi}\log\big(1+|x|^2|y|^2-2 \langle x,y\rangle \big)\hbox{ for }x,y\in B_1,$$  where $\langle\cdot,\cdot\rangle$ denotes the scalar product in $\R^2$.  

\begin{prop}\label{lililili} Let  the potential $V > 0$ belong to the class $C^1(\overline{B}_1)$ and let  $v_k\in C(\overline{B}_1)$ be a sequence of blow-up solutions  to problem  \eqref{eq0} satisfying \eqref{eqq0}. Then, 
along a subsequence
\beq\label{cuccu}\la_k|x|^2V(x)e^{v_k}\rightharpoonup 16\pi  {\bbm[\de]}_0\eeq
weakly in the measure sense, where ${\bbm[\de]}_0$ denotes the Dirac delta measure concentrated at the origin. Moreover,  
$$
v_k\longrightarrow 16\pi G(x,0)=8\log\frac{1}{|x|}
$$
in $C^1(K)$ for every compact set $K \subset \overline{B_1}\setminus\{0\}$.
\end{prop}

\

\subsection{{\large{Proof of Theorem \ref{th1} Part I: The necessary condition} }}
In this section, we exploit the integral identities obtained in Propositions \ref{propo1}-\ref{propo2} in order to derive necessary conditions for  the presence of bubbling solutions at $0$  for problem \eqref{eq0}--\eqref{eqq0}.

In the following we assume that  $v_k$ is a sequence satisfying \eqref{eq0}--\eqref{eqq0} and the potential $V$ verifies \eqref{eqq1}--\eqref{eqq2}.

We first establish a vanishing estimate for the gradient. 
\begin{prop}\label{necessary} The following holds: $$\frac{\partial v_k}{\partial x_i}(0)=o(1) ,\quad i=1,2$$
as $k\to +\infty$.

\end{prop}

\begin{proof} According to Proposition \ref{lililili} 
$$
v_k \longrightarrow 8\log\frac{1}{|x|}\quad \text{in } C^1(\partial B_1).
$$

Then we deduce  that
\beq\label{beda}\frac{\partial v_k}{\partial \nu}(x)\to  -8\hbox{ unif. for } x\in\partial B_1.\eeq
We have thus proved that 
$$\begin{aligned}\int_{\partial B_1} \bigg|\frac{\partial v_k}{\partial \nu}\bigg|^2x_i d\sigma&=64  \int_{\partial B_1}x_i+o(1)=o(1).
\end{aligned}
$$
Similarly
$$\begin{aligned}\int_{\partial B_1}\frac{\partial v_k}{\partial \nu}x_i d\sigma&=-8 \int_{\partial B_1}x_i d\sigma+o(1)=o(1).
\end{aligned}
$$
Moreover, by \eqref{cuccu} and using that $\nabla V(0)=0$ (see \eqref{eqq1}), we have
$$
 \la_k\int_{B_1} |x|^{2} \frac{\partial V}{\partial x_i}(x) e^{v_k}dx=o(1). $$ 
By inserting the above estimates into the identity provided by Proposition \ref{propo1} we deduce that $\frac{\partial v_k}{\partial x_i}(0)=o(1) $ for $i=1,2$.
\end{proof}

\

The  next result is a further consequence of the Pohozaev identities established in Proposition \ref{propo2} and, combined with the  previous proposition,  provides   alternative integral estimates for  blow-up solutions.  
\begin{prop}\label{necessary1} The following estimates hold:

 \beq\label{medio2}\begin{aligned}& \la_k\int_{B_1} e^{v_k}\bigg(\frac{\partial V}{\partial x_1}x_1-\frac{\partial V}{\partial x_2}x_2\bigg) dx  
=o(1),\end{aligned}\eeq

\beq\label{medio4}\begin{aligned}& \la_k\int_{B_1} e^{v_k}\bigg(\frac{\partial V}{\partial x_1}x_2+\frac{\partial V}{\partial x_2}x_1\bigg) dx  
=o(1) \end{aligned}\eeq as $k\to +\infty$.

\end{prop}

\begin{proof} 
Using the convergence \eqref{beda} as $k\to +\infty$ we get 

$$\int_{\partial B_1}\bigg|\frac{\partial v_k}{\partial \nu}\bigg|^2 (x_1^2-x_2^2) d\sigma=64\int_{\partial B_1} (x_1^2-x_2^2) d\sigma+o(1)=o(1)$$
thanks to $\int_{\partial B_1}(x_1^2-x_2^2)d\sigma=0$. On the other hand  $$\la_k\int_{\partial B_1}V(x)(x_1^2-x_2^2)d\sigma=O(\la_k)=o(1).$$  Inserting  the above two  estimates into the first Pohozaev identity of  Proposition \ref{propo2} written for $\la=\la_k$ and $v=v_k$ we find $$\int_{B_1} \la_k e^{v_k}\bigg(\frac{\partial V}{\partial x_1}x_1-\frac{\partial V}{\partial x_2}x_2\bigg) dx  =-\pi\bigg|\frac{\partial v_k}{\partial x_1}(0)\bigg|^2+\pi\bigg|\frac{\partial v_k}{\partial x_2}(0)\bigg|^2 .$$ Combining this identity with Proposition~\ref{necessary}, we deduce
\eqref{medio2}.

The proof of \eqref{medio4} follows analogously from the second Pohozaev
identity in Proposition~\ref{propo2}.

\end{proof}

\

Now we are in the position to give the proof of the first part of Theorem \ref{th1}, concerning the necessary condition.

According to the expansion \eqref{rmmm},  $$\frac{\partial V}{\partial x_1}(x)x_1-\frac{\partial V}{\partial x_2}(x)x_2= \gamma_1 x_1^2-\gamma_2 x_2^2 +O(|x|^3),$$ and
 $$\frac{\partial V}{\partial x_1}(x)x_2+\frac{\partial V}{\partial x_2}(x)x_2= (\gamma_1+\gamma_2  )x_1x_2 +O(|x|^3),$$
by which, using Proposition \ref{lililili},  we derive 
$$  \la_k\int_{B_1} e^{v_k}\bigg(\frac{\partial V(x)}{\partial x_1}x_1-\frac{\partial V(x)}{\partial x_2}x_2\bigg) dx=\la_k\int_{B_1}(\gamma_1 x_1^2-\gamma_2 x_2^2) e^{v_k}dx+o(1),$$
$$  \la_k\int_{B_1} e^{v_k}\bigg(\frac{\partial V(x)}{\partial x_1}x_2+\frac{\partial V(x)}{\partial x_2}x_2\bigg) dx=\la_k(\gamma_1+\gamma_2 )\int_{B_1}x_1x_2e^{v_k}dx+o(1).$$
Therefore Proposition \ref{necessary1} applies and yields\beq\label{exx}\la_k\int_{B_1}(\gamma_1 x_1^2-\gamma_2 x_2^2) e^{v_k}dx=o(1),\eeq
 \beq\label{exx1}\la_k(\gamma_1+\gamma_2 )\int_{B_1}x_1x_2e^{v_k}dx =o(1).\eeq

\medskip

We claim that 
\beq\label{lapla}\gamma_1+\gamma_2 \neq 0.\eeq Otherwise, by Proposition \ref{lililili} we would obtain 
$$\la_k\int_{B_1}(\gamma_1 x_1^2-\gamma_2 x_2^2) e^{v_k(x)}dx=\gamma_1 \la_k\int_{B_1}|x|^2 e^{v_k} dx =16\pi \gamma_1  +o(1).$$
Combining this estimate with \eqref{exx}, we would deduce $\gamma_1=\gamma_2=0$, which contradicts the nondegeneracy at the critical point $0$. 

Hence \eqref{lapla} holds, or, equivalently, $\Delta V(0)\neq 0$, which implies, by \eqref{exx1}, \beq\label{sera} \into \Im(x^2) e^{v_k} dx=2\into x_1x_2 e^{v_k} dx=o(1).\eeq Moreover, according to Theorem \ref{thweizhang},   the sequence $v_k$ must satisfy the simple blow-up property \eqref{simple}. 
In this case, following the approach in   \cite{gluck2012} and \cite{zhang2006}, one can perform  a refined analysis of blow-up solutions based on  Green’s representation formula and maximum principle argument,  in order to sharpen the result given in Proposition \ref{lililili}. 
 In particular one can prove that the solution
$v_k$ differs from a sequence of suitably scaled  “standard bubbles” \eqref{bubble} only by an 
$o(1)$  term.
More precisely, let $\beta_k\in B_1$ be the maximum point of $v_k$, namely: $$\beta_k\to 0 \, : \;v_k(\beta_k)=\max_{B_1} v_k(x)\to +\infty,$$
then Proposition \ref{lililili} can be improved to the following uniform estimate: 
$$v_k(x)= \log \frac{e^{v_k(\beta_k)}}{(1+ \frac{\la_k V(\beta_k)}{32 }e^{v_k(\beta_k)}|x^{2}-b_k|^2)^2}+o(1)\quad \hbox{ in }B_1,$$
 where $b_k=\beta_k^{2}$, or, equivalently,

\beq\label{cucucucucu}v_k(x)= \log \frac{32\tau_k}{(1+ \tau_k|x^{2}-b_k|^2)^2}-\log(\la_kV(\beta_k))+o(1)\quad \hbox{ in }B_1,\eeq
where\footnote{We use the notation $\sim$ to denote quantities that are of the same order as  $\la\to 0^+$. } $$\tau_k:=\frac{\la_k V(\beta_k)}{32}e^{v_k(\beta_k)}\sim\frac{1}{\la_k}$$ by evaluating \eqref{cucucucucu} for $x\in\partial B_1$ (where $v_k(x)=0$). Moreover, according to \eqref{simple},
$v_k(\beta_k)+2\log|b_k|+\log\la_k\leq C$ or, equivalently, $$\la_k|b_k|^2e^{v_k(\beta_k)}=O(1)$$ which implies, up to a subsequence,  $$\tau_k^{1/2}b_k\to \xi$$ for some $\xi=(\xi_1,\xi_2)\in\R^2$.

Hence, \eqref{cuccu}, \eqref{exx} and  \eqref{sera} can be rewritten as:
\beq\label{sera0} \tau_k\into   \frac{|x|^2}{(1+ \tau_k|x^{2}-b_k|^2)^2}  dx=\frac{\pi}{2}+o(1),\eeq 
 \beq\label{sera1}\tau_k\into   \frac{\gamma_1 x_1^2-\gamma_2 x_2^2}{(1+ \tau_k|x^{2}-b_k|^2)^2}  dx=o(1),\qquad\tau_k\into   \frac{\Im(x^2)}{(1+ \tau_k|x^{2}-b_k|^2)^2}  dx=o(1).\eeq
 On the other hand, by using first Remark \ref{remcopycopy} and next applying the change of variable  $z=\tau_k^{1/2} y$
$$\begin{aligned}\tau_k\into   \frac{\Im(x^2)}{(1+ \tau_k|x^{2}-b_k|^2)^2}  dx&=\frac{\tau_k}{2}\into  \frac{y_2}{|y|} \frac{dy}{(1+ \tau_k| y-b_k|^2)^2}  = \frac{1}{2}\int_{|z|\leq \tau_k^{\frac12}}  \frac{z_2}{|z|} \frac{dz}{(1+ | z-\tau_k^{1/2} b_k|^2)^2}  \\ &=\frac12\int_{\R^2}\frac{z_2}{|z|} \frac{dz}{(1+ | z-\xi|^2)^2}  +o(1)\end{aligned} $$
which implies by \eqref{sera1}

$$\int_{\R^2}\frac{z_2}{|z|} \frac{1}{(1+ | z-\xi|^2)^2}  dz=0$$
and, by Lemma \ref{signlemma}, $\xi_2=0$. Similarly
$$\begin{aligned}\tau_k\into   \frac{\Re(x^2)}{(1+ \tau_k|x^{2}-b_k|^2)^2}  dx
=\frac12\int_{\R^2}\frac{z_1}{|z|} \frac{1}{(1+ | z-(\xi_1,0)|^2)^2}  dz+o(1)\end{aligned} $$
Taking into account that   $\gamma_1 x_1^2-\gamma_2 x_2^2=\frac{\gamma_1-\gamma_2 }{2}|x|^2+\frac{\gamma_1+\gamma_2 }{2}\Re(x^2)$, by \eqref{sera0} we deduce
$$\begin{aligned}\tau_k\into   \frac{\gamma_1 x_1^2-\gamma_2 x_2^2}{(1+ \tau_k|x^{2}-b_k|^2)^2}  dx&=\frac\pi4(\gamma_1-\gamma_2 )+\frac{\gamma_1+\gamma_2 }{4}\int_{\R^2}\frac{z_1}{|z|} \frac{1}{(1+ | z-(\xi_1,0)|^2)^2}  dz+o(1)
\end{aligned}$$
which implies, by \eqref{sera1}
$$\pi(\gamma_1-\gamma_2 )+(\gamma_1+\gamma_2 )\int_{\R^2}\frac{z_1}{|z|} \frac{1}{(1+ | z-(\xi_1,0)|^2)^2}  dz=0.$$ Lemma \ref{signlemma} applies and gives $$\frac{\gamma_1^2-\gamma_2^2}{(\gamma_1+\gamma_2)^2}=\frac{\gamma_1-\gamma_2}{\gamma_1+\gamma_2}=-\frac{1}{\pi} I(\xi_1,0)\in (-1,1),$$ by which we derive $\gamma_1\cdot\gamma_2>0$;
moreover $$\gamma_1^2>\gamma_2^2 \Longleftrightarrow \xi_1<0,\quad \gamma_1^2<\gamma_2^2\Longleftrightarrow \xi_1>0,\quad \gamma_1 =\gamma_2 \Longleftrightarrow \xi_1=0.$$
\medskip

That proves the necessary condition stated in Theorem \ref{th1}.

\

 \begin{rmk}\label{remcopycopy} Let  $f:\R^2\to\R$ be such that $|x|^2f(x^2)\in L^1(\R^2)$. 
Then, using  the polar coordinates $(\rho,\theta)$ and then applying the change of variables $(\rho',\theta')=(\rho^{2},2\theta)$, we find the identity
$$\begin{aligned} \intr |x|^{2} f(x^{2}) dx&=\int_0^{+\infty} d\rho\int_0^{2\pi} \rho^{3}f(\rho^{2} e^{2{\rm i} \theta}) d\theta
\\ &= \frac14\int_0^{+\infty} d\rho'\int_0^{4 \pi}  \rho'f(\rho' e^{{\rm i}\theta'}) d\theta' \\ &=\frac{1}{2}\int_0^{+\infty} d\rho'\int_0^{2\pi}  \rho'f(\rho' e^{{\rm i}\theta'}) d\theta'
\\& =\frac{1}{2}\intr f(y) dy.\end{aligned}$$

\end{rmk}

\begin{lemma}\label{signlemma}
Let $\xi=(\xi_1,\xi_2)\in\R^2$ and define
$$
I(\xi):=\int_{\R^2}\frac{z_1}{|z|}\,\frac{1}{(1+|z-\xi|^2)^2}\,dz.
$$
Then $I\in C^2(\R),$ $I$ is odd and 
$$
I(\xi)=0 \quad \text{if and only if} \quad \xi_1=0,
$$
and moreover
$$
I(\xi_1,0)\to \pi  \ \text{ as } \xi_1\to +\infty,
\qquad
\partial_{\xi_1}I(\xi_1,0)>0 \;\;\forall \xi_1\in\R.
$$
\end{lemma}

\begin{proof}
We can write
$$
I(\xi)=\int_{\R} z_1\, F(z_1)\,dz_2,
$$
where
$$
F(z_1):=\int_{\R}
\frac{1}{\sqrt{z_1^2+z_2^2}}\,
\frac{1}{\bigl(1+(z_1-\xi_1)^2+(z_2-\xi_2)^2\bigr)^2}
\,dz_2 >0.
$$

By symmetry  we obtain
$$
I(\xi)=\int_0^{+\infty} z_1\,[F(z_1)-F(-z_1)]\,dz_1.
$$

Assume now that $\xi_1>0$. For every $z_1>0$ and every $z_2\in\R$ we have
$$
(z_1-\xi_1)^2<(-z_1-\xi_1)^2,
$$
which implies
$$
\frac{1}{\bigl(1+(z_1-\xi_1)^2+(z_2-\xi_2)^2\bigr)^2}
>
\frac{1}{\bigl(1+(-z_1-\xi_1)^2+(z_2-\xi_2)^2\bigr)^2}.
$$
Integrating with respect to $z_2$, we deduce that
$$
F(z_1)>F(-z_1)
\qquad \text{for every } z_1>0.
$$
Therefore,
$$
I(\xi)=\int_0^{+\infty} z_1\,[F(z_1)-F(-z_1)]\,dz_1>0.
$$

  The case $\xi_1<0$ follows by the same argument with reversed inequalities,
while if $\xi_1=0$ the function $F$ is even and hence $I(\xi)=0$.
Let us consider
$$
I(\xi_1,0):=\int_{\R^2}\frac{z_1}{|z|}\,\frac{1}{(1+|z-(\xi_1,0)|^2)^2}\,dz=\int_{\R^2}\frac{z_1+\xi_1}{|z+(\xi_1,0)|}\,\frac{1}{(1+|z|^2)^2}\,dz .
$$ So by dominated convergence we get $$I(\xi_1,0)\to \intr \frac{1}{(1+|z|^2)^2}\,dz=\pi \hbox{ as }\xi_1\to +\infty.$$ Moreover 
$$\partial_{\xi_1} I(\xi_1,0)= \intr\frac{z_2^2}{|z-(\xi_1,0)|^3}\frac{1}{(1+|z|^2)^2}\,dz>0.$$
\end{proof}

 \begin{rmk}\label{fails} We point out that if one attempts to apply our technique to $N\geq 2$, a major difficulty arises
from the role of the Pohozaev identity. Indeed, the analogous of Proposition \ref{necessary1} would yield
$$\la_k\int_{B_1}  e^{v_k}\bigg(\frac{\partial V}{\partial x_1}\Re(x^N)-\frac{\partial V}{\partial x_2}\Im(x^N)\bigg) dx=o(1),$$
$$\la_k\int_{B_1}  e^{v_k}\bigg(\frac{\partial V}{\partial x_1}\Im(x^N)+\frac{\partial V}{\partial x_2}\Re(x^N)\bigg) dx=o(1).$$

Observe that only in the case $N=1$  can the above integrals be estimated effectively, whereas for $N\geq 2$    the limited information provided by the Pohozaev
identity  is not sufficient to control  the second derivatives of the coefficient functions. Much more accurate approximate solutions would be required in order  to reduce the error terms.   
 This explains why our method, based on the combination of two Pohozaev identities, 
fails to provide the necessary condition stated in Theorems \ref{th1} for $N\geq 2$. 
\end{rmk}

{\Large{\section{\bf{Proof of Theorem \ref{th1} Part II: Construction \\of Bubbling Solutions}}}}

During this section we assume that the potential $V$ satisfies assumptions \eqref{eqq1}--\eqref{eqq2} and, in addition, \eqref{pot}; then   we will construct a sequence of bubbling solutions for the problem \eqref{eq0}--\eqref{eqq0}.

The second part of Theorems \ref{th1} will be a consequence of  more general results concerning Liouville-type problem \eqref{eq1}. In order to provide such  results, we now give a construction of a suitable approximate solution. 

\

\subsection{{\large{Approximate solution}}} 
Let us define the \textit{bubble}
$$W_\la=W_{\la, b}=
\log  {\la\over (\frac{\la}{32}+|x^{2}- b|^2)^2}\,$$ which corresponds to a solution of the limit problem \eqref{limit} for $N=1$.

To obtain a better first approximation, we need to modify the functions $W_\la$   in order to satisfy the zero boundary condition. Precisely, we consider the projections $P W_\la $ onto the space $ H^1_0(B_1)$ of
$W_\la$, where the projection  $P:H^1(\R^N)\to  H^1_0(B_1)$ is
defined as the unique solution of the problem
$$
 \Delta P v=\Delta v\quad \hbox{in}\ B_1,\qquad  P v=0\quad \hbox{on}\ \partial B_1.
$$

Consider ${\cal H}(x,y)$ the  regular part of the Green function in $B_1$:
\beq\label{regG}{\cal H}(x,y)=\frac{1}{2\pi}\log\Big(|x|\Big|y-\frac{x}{|x|^2}\Big|\Big)=\frac{1}{4\pi}\log(|x|^2|y|^2+1-\langle x,y\rangle),
\eeq  where $\langle\cdot,\cdot\rangle$ denotes the scalar product in $\R^2$. 
Let us take $b$ in a small neighborhood of $0$: observe that  the function ${\cal H}(x^2,b)$ is harmonic in $B_1$ and satisfies ${\cal H}(x^2,b)=\frac{1}{2\pi}\log |x^2-b|$ on $\partial B_1$. A straightforward computation gives that for any $x\in\partial B_1$ 
$$\big|PW_\la- W_\la+\log\la-8\pi{\cal H}(x^2,b)\big|=\big|W_\la-\log\la+4\log |x^2-b|
\big|\leq C\la$$
uniformly for $b$ in a small neighborhood of $0$. 
Since the expressions considered inside the absolute values 
are harmonic in $B_1$, then the maximum principle applies and implies
the following asymptotic expansion
\beq\label{pro-exp1}\begin{aligned}
 PW_\la=& W_\la-\log\la+8\pi{\cal H}(x^2,b)+O(\la)\\ =&-2\log\(\frac{\la}{32}+|x^{2}- b|^2\)+8\pi{\cal H}(x^2,b)+O(\la)
\end{aligned}\eeq 
uniformly for $x\in \overline B_1$ and $b $ in a small neighborhood of $0$.

We shall look for a solution to \eqref{eq1} in a small neighborhood of the first approximation, namely a solution of the form
 $$v_\la=PW_\la
 + {\phi}_\la,$$ where the rest term
$\phi_\la$ is small in
$H^1_0(B_1)$-norm.

We are now in the position to state  the main theorem of second part of the paper.

\begin{thm} \label{main1} Assume that $V$ satisfies the hypotheses \eqref{eqq1}-\eqref{eqq2} and, in addition, \eqref{pot}. Then, for $\la$ sufficiently small  there
exist  $\phi_\la \in H^1_0(B_1)$ and $b=b_\la=O(\sqrt\la)$ such that the couple $
PW_\la+\phi_\la$ solves problem
\eqref{eq1}.
Moreover, for any fixed $\e>0$,
$$\| \phi_\la \|_{H^1_0(B_1)}\leq \la^{\frac34-\e}\;\;\hbox{ for } \la \hbox{ small enough}.$$
\end{thm}

In the remaining part of the paper we will prove Theorem \ref{main1} and at the end of the section  we shall see how the second part of Theorem \ref{th1} follows quite directly as a corollary. 

Before the proof, we now set the notations and basic well-known
facts to be used in the rest of the paper. We denote by  $\|\cdot\|$ and $\|\cdot\|_p$  the norms in  the space $ H^1_0(B_1)$ and $L^p(B_1)$, respectively, namely
 $$\|u\|:=\|u\|_{ H^1_0(B_1)}
 ,\qquad \|u\|_p:=\|u\|_{L^p(B_1)}
 \quad \forall u\in  H^1_0(B_1).$$

In next lemma we recall the well-known Moser-Trudinger inequality
(\cite{Moe, Tru}).

\begin{lemma}\label{tmt} There exists $C>0$ such that 
 $$\int_{B_1} e^{\frac{4\pi u^2}{\|u\|^2}}dy\le C \quad \forall u\in{ H}^1_0(B_1).$$ 
 In particular,  for any $q\geq 1$
 $$\| e^{u}\|_{q}\le  
 C_q e^{{\frac{q}{ 16\pi}}\|u\|^2}\quad \forall  u\in{H}^1_0(B_1).$$

\end{lemma}

As commented in the introduction, our proof uses the singular
perturbation method. For that, the nondegeneracy of the functions
that we use to build our approximating solution is essential. The following
proposition is devoted to the nondegeneracy of the finite mass 
solutions of the singular  Liouville equation.

\begin{prop}
\label{esposito} Assume that $\phi:\R^2\to\R$  solves the problem
\begin{equation}\label{l1}
-\Delta \phi =32{|y|^{2}\over (1+|y^{2}-\xi|^2)^2}\phi\;\;
\hbox{in}\ \rr^2,\quad \int_{\R^2}|\nabla
\phi(y)|^2dy<+\infty.
\end{equation}
 Then there exist $c_0,\,c_1,\, c_2\in\R$
such that
$$\phi(y)=c_0  Z_0+ c_1Z_1 +c_2Z_2$$ where
$$Z_0(y):   = {1-|y^{2}-\xi|^2\over 1+|y^{2}-\xi|^2} ,\ \; \;Z_1(y):={ \Re(y^{2}-\xi)\over  1+|y^{2}-\xi|^2} 
,\ \;\;Z_2(y):={ \Im(y^{2}-\xi)\over  1+|y^{2}-\xi|^2}.
$$ \end{prop}
\begin{proof}
In \cite[Theorem 6.1]{gpistoia} it was proved that any solution
$\phi$ of \eqref{l1} is actually a bounded solution. Therefore we
can apply the result in \cite{dem} to conclude that $\phi= c_0 \phi_0 + c_1
\phi_1 + c_2 \phi_2$ for some $c_0,c_1,c_2\in \R$.

\end{proof}

\subsection{{\large{Estimate of the error term}}}
The goal of this section is to  provide an estimate of the error up to which the function $PW_\la$ solves problem \eqref{eq1}.
First of all, we perform the following estimates.

From now on, let 
$M>0$ be a sufficiently large number, to be chosen later.

\begin{lemma}\label{aux00}   Let
 $s=0,1,2,3$ and $p\geq 1$ be fixed. The following holds:\beq\label{aux00esti1}
\|  |x|^{2+s}e^{W_{\la}}\|_p
\leq C\la^{\frac{s}{4}}\la^{-\frac{p-1}{2p}} ,\quad \|  \la |x|^{2+s} e^{PW_{\la}}\|_p
\leq C\la^{\frac{s}{4}}\la^{-\frac{p-1}{2p}} \eeq
uniformly for $|b|\leq M\sqrt{\la}$. 
\end{lemma}
\begin{proof}
We compute $$\begin{aligned}\||x|^{2+s} e^{W_{\la}}\|_p^p&=
\la^{p}\into\frac{|x|^{(2+s )p}}{(\frac{\la}{32}+|x^2-b|^2)^{2p}}dx
\leq C\la^{ s \frac{p}{4}-\frac{p-1}{2}} \intr\frac{|z|^{(s +2)p}}{(1+|z|^4)^{2p}}dz
\end{aligned}$$ uniformly for $|b|\leq M\sqrt\la$.
Taking into account that the last integral is finite for   $s=0,1,2,3$ and $p\geq 1$ 
we deduce the first part of  \eqref{aux00esti1}.
To prove the second part it is sufficient to observe that 
by \eqref{pro-exp1} 
we derive
\beq\label{poi}\la e^{PW_{\la}}= e^{W_{\la}+O(1)}
=e^{W_{\la}}(1+O(1)).\eeq
\end{proof}

\begin{lemma}\label{aux}  Define
$${R}_\la:=
-\Delta PW_\la-\la V(x) |x|^{2}e^{PW_\la}= |x|^{2}e^{W_\la}-\la V(x)|x|^{2} e^{PW_\la}.$$
 For any fixed $p\geq 1$ the following holds

$$\|R_\la\|_{p}=O\big(\la^{\frac34-\frac{p-1}{ 2p}}\big).$$ uniformly for $|b|\leq M\sqrt{\la}$.  
Moreover the following expansion holds  \beq\label{cuore}R_\la=(\gamma_1x_1^2+\gamma_2x_2^2)|x|^{2}e^{W_\la}+\big(O(|x|^3)+O(\sqrt\la|x|^2)+O(\la)\big)|x|^{2}e^{W_\la}\eeq uniformly for $|b|\leq M \sqrt{\la}$.  Here $\gamma_1, \gamma_2$ have been defined in \eqref{rmmm}.

\end{lemma}

\begin{proof}
By \eqref{regG} and  \eqref{pro-exp1}  
we derive $$\begin{aligned}\la V(x)|x|^{2} e^{PW_\la} &
=V(x)|x|^{2} e^{W_\la +8\pi {\cal H}(x^2,b)+O(\la )} 
\\ &=V(x)|x|^{2}e^{W_\la}(1+|x|^4|b|^2-2\Re( \overline bx^2))^2+O(\la)|x|^{2}e^{W_\la}
\\ &= V(x)|x|^{2}e^{W_\la}+  \big(O(|b||x|^2)+O(\la)\big)|x|^{2}e^{W_\la}
\end{aligned}$$ uniformly for $b$ in a small neighborhood of $0$. 
According to hypothesis \eqref{pot} we deduce
$$\begin{aligned}\la V(x)|x|^{2} e^{PW_\la}  
=(\gamma_1x_1^2+\gamma_2x_2^2)|x|^{2}e^{W_\la}+\big(O(|x|^3)+O(|b||x|^2)+O(\la)\big)|x|^{2}e^{W_\la}
\end{aligned}$$  uniformly for $b$ in a small neighborhood of $0$. 
The conclusion follows by Lemma \ref{aux00}.

\end{proof}

\subsection{{\large{Analysis of the linearized operator}}}
According to Proposition \ref{esposito}, by the change of variable $x=(\frac{\la}{32})^{\frac14} y$, we immediately get that  all solutions $\psi $  of
$$
-\Delta \psi= {\la|x|^{2}\over (\frac{\la}{32}+|x^{2 }-b|^2)^2}\psi =|x|^{2} e^{W_\la}\psi\quad \hbox{in}\;\; \rr^2,\qquad \intr |\nabla \psi|^2 dx<+\infty$$
are linear combinations of the functions
$$Z^0_{\la}(x)={\frac{\la}{32}-|x^{2}-b|^2\over \frac{\la}{32}+|x^{2 }-b|^2},\ Z^1_{\la}(x)=
{ \sqrt{\frac{\la}{32}}\,\Re(x^{2 }-b)\over  \frac{\la}{32}+|x^{2 }-b|^2},\ Z^2_{\la}(x)=
{ \sqrt{\frac{\la}{32}} \,\Im(x^{2 }-b)\over  \frac{\la}{32}+|x^{2 }-b|^2}
.$$
We introduce their projections $PZ^j_{\la}$ onto $H^1_0(B_1).$ It is immediate that
 \begin{equation}\label{pz0}
PZ^0_{\la}(x)=Z^0_{\la}(x)+1+ O(\la).
\end{equation}
 and
  \begin{equation}\label{pzi}
PZ^j_{\la}(x)=Z^j _{\la}(x) + O\big(\sqrt{\la}\big),\;\; j=1,2
\end{equation}
uniformly with respect to $x\in\overline B_1$ and $b$ in a small neighborhood of $0$.

Let us consider the following linear problem: given
$ h\in H_0^1(B_1)$,  find a function $\phi\in  H^1_0(B_1)$ and constants $c_1,c_2\in\R$ satisfying
\begin{equation}\label{lla2}
\left\{\begin{aligned}&-\Delta \phi   -\la V(x)|x|^{2} e^{P{W}_\la}\phi=\Delta h+\sum_{j=1,2}c_j Z^j_{\la} |x|^{2}e^{W_{\la}}\\ &\into \nabla \phi\nabla PZ_\la^j=0\;\;j=1,2\end{aligned}\right..
\end{equation}

In order to solve problem \eqref{lla2}, we need to establish an a priori estimate. For the proof we refer to \cite{delkomu} (Proposition 3.1) or \cite{espogropi} (Proposition 3.1). 

\begin{prop}\label{inv} 
There exist $\lambda_0>0$ and $C>0$ such that for any $\la \in(0, \la_0)$, any $b\in \R^2$ with $|b|<M\sqrt\la$ and   any $h\in  H^1_0(B_1)$, if  
$(\phi,c_1,c_2)\in  H^1_0(B_1)\times \R^2$  solves  \eqref{lla2}, then the following holds $$\|\phi\|    \leq C  |\log\la |    \|h\| .$$
\end{prop}

\

Next, for any $ p>1,$ let $$i^*_{p}:L^{p}(B_1)\to H^1_0(B_1)$$ be the
adjoint operator of the embedding
$i_{p}:H^1_0(B_1)\hookrightarrow L^{p\over p-1 }(B_1),$ i.e.
$u=i^*_{p}(v)$ if and only if $-\Delta u=v$ in $B_1,$ $u=0$ on
$\partial B_1.$ We point out that $i^*_{p}$ is a continuous
mapping, namely
\begin{equation}
\label{isp} \|i^*_{p}(v)\| \le c_{p} \|v\|_{p}, \ \hbox{for any} \ v\in L^{p}(B_1),
\end{equation}
for some constant $c_{p}$ which depends on $p.$
Next let us set
$$      {K }:=\hbox{span}\left\{PZ^1_{\la},\ PZ^2_{\la}\right\}$$
and
$$      {K ^\perp }:= \left\{\phi\in H^1_0(B_1)\ :\ \into \nabla \phi \nabla PZ^1_{\la}dx= \into \nabla \phi \nabla PZ^2_{\la}dx= 0\right\} $$
and  denote by
$$     \Pi : H^1_0(B_1)\to       {K },\qquad      {\Pi ^\perp}: H^1_0(B_1)\to       {K ^\perp }$$
the corresponding projections.
Let $      L: K^\perp \to K^\perp$  be the
linear operator defined by
\beq\label{elle}
      L(     \phi):=  \Pi^{\perp}\Big(   {i^*_{p}}\big(       \la V(x)|x|^{2}e^{PW_\la}     \phi \big)\Big) - \phi. 
\eeq
Notice that problem \eqref{lla2} reduces to $$L(\phi)=\Pi^\perp h, \quad \phi\in K^\perp.$$
 
 As a consequence of Proposition \ref{inv} we derive the invertibility of $L$.

\begin{prop}\label{ex} For any $p>1$ there exist $\lambda_0>0$ and $C>0$ such that for any $\la \in(0, \la_0)$, any $b\in\R^2$ with $|b|<M\sqrt\la$ and 
 any $h\in K^\perp$  there is a unique solution $ \phi\in K^\perp$ to the problem $$L(\phi)=h.$$ In
particular, $L$ is invertible; moreover, $$\| L^{-1} \| \leq C
|\log \lambda |.$$

\end{prop}
\begin{proof}  Observe that the operator $\phi\mapsto \Pi^\perp\big( {i^*_{p}}\(       \la V(x)|x|^{2}e^{PW_\la}     \phi \)\big)$ is a compact operator in $K^\perp$.
   Let us consider the case $h=0$, and take $\phi\in K^\perp$ with $L(\phi)=0$. In other words,  $\phi$ solves
the system \eqref{lla2} with $h=0$ for some $c_1,c_2\in\R$. Proposition \ref{inv} implies $\phi\equiv 0$. Then, Fredholm's alternative implies the existence and uniqueness result.

Once we have existence, the norm estimate follows directly from Proposition \ref{inv}.
\end{proof}

\subsection{{\large{The nonlinear problem: a contraction argument}}}
In order to find a blowing up sequence for problem \eqref{eq0}--\eqref{eqq0} let us consider the following intermediate problem:

\beq\label{inter}\left\{\begin{aligned}&-\Delta(PW_\la+\phi)-\la V(x) |x|^{2}e^{PW_\la+\phi}=\sum_{j=1,2}c_j Z_\la^j|x|^{2} e^{W_\la},\\ &\phi \in H^1_0(B_1),\;\;\;\; \into \nabla \phi\nabla PZ_\la^jdx=0 \; \; j=1,2.\end{aligned}\right.\eeq
 
 Then it is convenient to solve as a first step the problem for $\phi$ as a function of $b$. 
 To this aim, first let us rewrite problem \eqref{inter} in a more convenient way. 
 
 In what follows we denote by $N:K^\perp\to K^\perp$ the nonlinear operator 
$$N(\phi)=\Pi^\perp\({i^*_{p}}\big(       \la V(x)|x|^{2}e^{PW_\la}(e^{\phi}-1-    \phi) \big)\).$$
 Therefore problem \eqref{inter} turns out to be equivalent to the problem
 
 \beq\label{interop}  L(\phi)+N(\phi)=\tilde R,\quad \phi\in K^\perp\eeq
 where, recalling Lemma \ref{aux},  $$\tilde R=\Pi^\perp\({i^*_{p}}\big(R_\la\big)\)=
  \Pi^\perp\(PW_\la -{i^*_{p}}\big(\la|x|^{2}V(x)e^{P W_\la}\big)\).$$
 
 We need the following auxiliary lemma.
 
 \begin{lemma}\label{auxnonl} 
 For any $p> 1$ there exists $\la_0>0$ such that   for any $\la\in (0,\la_0)$, any $b\in \R^2$ with $|b|\leq M \sqrt\la$  and any $\phi_1,\phi_2\in H_0^1(B_1)$ with $\|\phi\|_1,\,\|\phi_2\|<1$ the following holds
\beq\label{skate1}\|e^{\phi_1}-\phi_1-e^{\phi_2}+\phi_2\|_p\leq C\big(\|\phi_1\|+\|\phi_2\|\big)\|\phi_1-\phi_2\|,\eeq
 \beq\label{skate2}\|N(\phi_1)-N(\phi_2)\|\leq C\la^{\frac{1-p^2}{ 2p^2}}\big(\|\phi_1\|+\|\phi_2\|\big)\|\phi_1-\phi_2\|.\eeq
 
 \end{lemma}
 \begin{proof}
A straightforward computation gives that  the inequality $|e^a-a-e^b+b|\leq e^{|a|+|b|}(|a|+|b|)|a-b|$ holds  for all $a,b\in \R$. Then, by applying H\"older's inequality with $\frac1q+\frac1r+\frac1t=1$, we derive
$$\|e^{\phi_1}-\phi_1-e^{\phi_2}+\phi_2\|_p\leq C\|e^{|\phi_1|+|\phi_2|}\|_{pq}\big(\|\phi_1\|_{pr}+\|\phi_2\|_{pr}\big)\|\phi_1-\phi_2\|_{pt}$$
 and \eqref{skate1} follows by using Lemma \ref{tmt} and the continuity of the embeddings $H^1_0(B_1)\subset L^{pr}(B_1)$ and $H^1_0(B_1)\subset L^{pt}(B_1)$. 
  Let us prove \eqref{skate2}. According to \eqref{isp}  we get
 $$\|N(\phi_1)-N(\phi_2)\|\leq C\|\la V(x)|x|^{2} e^{PW_\la}(e^{\phi_1}-\phi_1-e^{\phi_2}+\phi_2)\|_p,$$
and by H\"older's inequality with $\frac1p+\frac1q=1$ we derive 
$$\begin{aligned}\|N(\phi_1)-N(\phi_2)\|&\leq C\|\la |x|^{2}e^{PW_\la}\|_{p^2}\|e^{\phi_1}-\phi_1-e^{\phi_2}+\phi_2|\|_{pq}\\ &\leq C\|\la |x|^{2}e^{PW_\la}\|_{p^2}\big(\|\phi_1\|+\|\phi_2\|\big)\|\phi_1-\phi_2\|\end{aligned}
$$ by \eqref{skate1}, and  the conclusion follows recalling Lemma \ref{aux00}. 
 \end{proof}
Problem \eqref{inter} or, equivalently, problem \eqref{interop}, turns out to be solvable for any choice of point $b$ with $|b|\leq M\sqrt\la$, provided
that $\la$ is sufficiently small. Indeed we have the following result.

\begin{prop}\label{nonl} 
For any $\e\in (0,\frac14)$ there exists $\la_0>0$ such that for any $\la\in (0,\la_0)$ and any $b\in \R^2$ with $|b|<M\sqrt\la$ there is a unique $\phi_\la=\phi_{\la,b}\in K^\perp$ satisfying \eqref{inter} for some $c_1,c_2\in \R$ and  
$$\|\phi_{\la}\|\leq \la^{\frac34-\e}.$$ 
\end{prop}
\begin{proof} Since, as we have observed, problem \eqref{interop} is equivalent to problem \eqref{inter}, 
we will show that problem \eqref{interop} can be solved via a contraction mapping argument. Indeed, in virtue of Proposition \ref{ex}, let us introduce the map
$$T:=L^{-1}(\tilde R-N(\phi)),\quad \phi\in K^\perp.$$
 Let us fix $$0<\eta<\e$$ and $p>1$ sufficiently close to 1. According to \eqref{isp} and Lemma \ref{aux} we have
\beq\label{non1}\|\tilde R\|=O(\la^{\frac34-\eta}).\eeq
Similarly, by \eqref{skate2}, \beq\label{non2}\|N(\phi_1)-N(\phi_2)\|\leq C\la^{-\eta}(\|\phi_1\|+\|\phi_2\|)\|\phi_1-\phi_2\|\quad \forall \phi_1,\phi_2\in H_0^1(B_1), \|\phi_1\|,\|\phi_2\|<1.\eeq In particular, by taking $\phi_2=0$, 
\beq\label{non3}\|N(\phi)\|\leq C\la^{-\eta}\|\phi\|^2\quad \forall \phi\in H_0^1(B_1), \|\phi\|<1.\eeq

We claim that $T$ is a contraction map over the ball $$\Big\{\phi\in K^\perp\,\Big|\, \|\phi\|\leq \la^{\frac34-\e}\Big\}$$ provided that $\la$ is small enough. Indeed, combining  Proposition \ref{ex}, \eqref{non1}, \eqref{non2}, \eqref{non3} with the choice of $\eta$,
we have 
$$\|T(\phi)\|\leq C|\log\la|(\la^{\frac34-\eta}+\la^{-\eta}\|\phi\|^2)<\la^{\frac34-\e},$$
$$\begin{aligned}\|T(\phi_1)-T(\phi_2)\|&\leq C|\log\la|\|N(\phi_1)-N(\phi_2)\|&\leq C\la^{-\eta}|\log\la| (\|\phi_1\|+\|\phi_2\|)\|\phi_1-\phi_2\|\\ &<\frac12\|\phi_1-\phi_2\|.\end{aligned}$$

\end{proof}

\subsection*{{\large{Proof of Theorems \ref{main1}}}} 
After problem \eqref{inter} has been solved according to Proposition \ref{nonl}, then we find a solution to the original problem \eqref{eq1} if $b$   is such that $$c_j=0\hbox{ for }j=1,2.$$ 
 Let us find the condition satisfied by $b$ in order to get the $c_j$'s equal to zero. 

We multiply the equation in \eqref{inter} by   $PZ_\la^j$ and integrate over $B_1$:
\beq\label{masca}\begin{aligned}\into \nabla (PW_\la+\phi_{\la}) \nabla PZ^j_{\la} dx &-\la \into V(x) |x|^{2}e^{P W_\la+\phi_{\la}}PZ^j_{\la} dx\\&=\sum_{h=1,2}c_h \into Z_\la^h|x|^{2} e^{W_\la}PZ_\la^j dx.\end{aligned}
\eeq
The object  is now to expand each integral of the above identity and analyze the  leading term.
 Let us begin by observing that  the orthogonality in \eqref{inter} gives
\beq\label{masca1} \into \nabla \phi_{\la} \nabla PZ^j_{\la} dx =\into |x|^{2} e^{W_\la} \phi_{\la}Z_\la^j dx=0\eeq
and, using  \eqref{pzi}, by a direct computation we obtain\beq\label{masca2}\into Z_\la^h|x|^{2} e^{W_\la}PZ_\la^j dx= \intr Z_\la^jZ_\la^h|x|^{2} e^{W_\la}dx+o(1)=\left\{\begin{aligned}&\frac43 {\pi}+o(1) &\hbox{ if }h=j\\ &o(1)&\hbox{ if }h\neq j\end{aligned}\right..\eeq
Using the expansion \eqref{cuore}  and Lemma \ref{aux00} 
we get
\beq\label{coll}\begin{aligned}\int_{B_1} \nabla PW_\la &\nabla PZ^j_{\la} dx -\la \int_{B_1} V(x) |x|^{2}e^{P W_\la}PZ^j_{\la} dx\\ &=
-\int_{B_1} R_\la PZ^j_{\la} dx\\ &=
-\int_{B_1} |x|^{2}e^{W_\la} \Big(\gamma_1 x_1^2+\gamma_2 x_2^2+O(|x|^3)+O(\sqrt\la|x|^2)+O(\la) \Big)PZ^j_{\la} dx
\\ &= -\int_{B_1} |x|^{2}(\gamma_1 x_1^2+\gamma_2 x_2^2)e^{W_\la}PZ^j_{\la} dx+ O(\la^\frac34)
\\ &=-
\intr  |x|^{2}(\gamma_1 x_1^2+\gamma_2 x_2^2)e^{W_\la}Z^j_{\la} dx
+ O(\la^{\frac34})\end{aligned}\eeq uniformly for $|b|\leq M\sqrt\la$, where for the last estimates we have used \eqref{pzi}.

Let us pass to examine  the term containing the function $\phi_\la$: recalling the orthogonality \eqref{masca1} and Lemma \ref{aux}
\beq\label{mos}\begin{aligned} \la\int_{B_1}|x|^2 V(x) e^{P W_\la}(e^{\phi_\la}-1)PZ^i_{\la} dx&= -\into R_\la
(e^{\phi_\la}-1)PZ^i_{\la} dx
\\ &\;\;\;\;+ \int_{B_1}  |x|^2e^{W_\la}(e^{\phi_\la}-1-\phi_\la)PZ^i_{\la}  dx\\ &\;\;\;\;+ \into  |x|^2e^{W_\la}\phi_\la(PZ^i_{\la}-Z_\la^i)  dx.
\end{aligned}\eeq
Now, let us fix $\e>0$ sufficiently small and $p>1$ sufficiently close to 1.  
Next let  $1<q<\infty$ be such that  $\frac1p+\frac1q=1$. Then, 
\eqref{skate1} with $\phi_2=0$ and Proposition \ref{nonl} give
$$\|e^{\phi_\la}-1-\phi_\la\|_q\leq  C\|\phi_\la\|^2\leq C \la^{\frac32-2 \e} 
$$
and, consequently, 
\beq\label{tmtsur}\|e^{\phi_\la}-1\|_q\leq C\|\phi_\la\|\leq C\la^{\frac34-\e}
.\eeq
Therefore,  Lemma \ref{aux00}  implies
 \beq\label{capra1}\begin{aligned} 
 \int_{B_1} |x|^2e^{W_\la}(e^{\phi_\la}-1-\phi_\la)PZ^i_{\la} dx&
= O(\|  |x|^2e^{ W_\la}(e^{\phi_\la}-1-\phi_\la)\|_1)\\ &=
O(\|   |x|^2e^{W_\la} \|_p\|e^{\phi_\la}-1-\phi_\la\|_q)\\ &=O\Big(\la^{\frac32-\frac{p-1}{2p}-2\e} \Big).
\end{aligned}\eeq
Now, by Lemma \ref{aux00}
 \beq\label{capra2}\begin{aligned}  \into R_\la (e^{\phi_\la}-1) 
 PZ^i_{\la} dx&=
O\big(\big\|  
R_\la (e^{\phi_\la}-1)\big\|_1\big)
=O\big(\big\| 
R_\la\|_p\|e^{\phi_\la}-1\|_q\big)\\ &= O\Big(\la^{\frac32-\frac{p-1}{2p}-\e}\Big).\end{aligned}\eeq
Finally by \eqref{pzi}   and  Lemma \ref{aux00}
 \beq\label{capra3}\begin{aligned}  \into  |x|^2e^{W_\la}\phi_\la(PZ^i_{\la}-Z_\la^i)  dx&=O(\sqrt\la) \int_{B_1} |x|^2e^{W_\la}|\phi_\la|  dx
 \\ &=O(\sqrt\la\||x|^2e^{W_\la}\|_p\|\phi_\la\|)
 \\ &=O(\la^{\frac12-\frac{p-1}{2p}+\frac34-\e}). \end{aligned}\eeq
 By inserting \eqref{capra1}-\eqref{capra2}-\eqref{capra3} into \eqref{mos}, 
 we obtain
 \beq\label{capra}\la \int_{B_1} V\big(x\big) |x|^2e^{P W_\la}(e^{\phi_\la}-1)PZ^i_{\la} dx= o(\la) \eeq  uniformly for $|b|\leq M\sqrt\la$, provided that that $\e$ is chosen sufficiently close to $0$ and $p$ sufficiently close to $1$.

In order to conclude, combining \eqref{masca1}, \eqref{masca2}, \eqref{coll} and  \eqref{capra},  
 the identities \eqref{masca} 
can be rewritten as
\beq\label{assem}\begin{aligned} &-\intr |x|^2(\gamma_1 x_1^2+\gamma_2 x_2^2)e^{W_\la}Z_\la^j dx+o(\sqrt\la)
=\frac{4}{3}\pi c_j+o(c_1)+o(c_2), \quad j=1,2\end{aligned}\eeq uniformly for $|b|\leq M\sqrt\la$.
Let us observe that, by applying the change of variable first $x=(\frac{\la}{32})^{\frac14}y$ and next using Remark \ref{remcopycopy},  $$\begin{aligned}&\intr |x|^2(\gamma_1 x_1^2+\gamma_2 x_2^2)e^{W_\la}Z_\la^1 dx\\&=\frac12\intr |x|^2\big((\gamma_1-\gamma_2)\Re(x^2)+(\gamma_1+\gamma_2)|x|^2\big)e^{W_\la}Z_\la^j dx
\\ &
=\frac{\la\sqrt\la}{8\sqrt2} \intr |x|^2\frac{(\gamma_1-\gamma_2)\Re(x^2)+(\gamma_1+\gamma_2)|x|^2}{(\frac{\la}{32}+|x^2-b|^2)^3}\Re(x^2-b)dx
\\ & =2\sqrt{2\la} \intr |y|^2\frac{(\gamma_1-\gamma_2)\Re(y^2)+(\gamma_1+\gamma_2)|y|^2}{(1+|y^2-\sqrt{32}\la^{-\frac12}b|^2)^3}\Re(y^2-\sqrt{32}\la^{-\frac12}b)dy
\\ &=\sqrt{2\la}\intr \frac{(\gamma_1-\gamma_2)z_1+(\gamma_1+\gamma_2)|z|}{(1+|z-\sqrt{32}\la^{-\frac12}b|^2)^3}(z_1-\sqrt{32}\la^{-\frac12}b_1)dz
\\ &=(\gamma_1-\gamma_2)\pi\frac{\sqrt{2\la}}{4}+\sqrt{2\la}(\gamma_1+\gamma_2){\cal F}_1(\sqrt{32}\la^{-\frac12}b) \end{aligned}
$$ where we have used that $\intr \frac{z_1^2}{(1+|z|^2)^3}dz=\frac{\pi}{4}$ and ${\cal F}=({\cal F}_1,{\cal F}_2)$ is the vector field defined in Corollary \ref{finalcor}.  Similarly
$$\begin{aligned}&\intr |x|^2(\gamma_1 x_1^2+\gamma_2 x_2^2)e^{W_\la}Z_\la^2 dx=\sqrt{2\la}(\gamma_1+\gamma_2){\cal F}_2(\sqrt{32}\la^{-\frac12}b).\end{aligned}
$$
Therefore, \eqref{assem} admits the following equivalent formulation:
\beq\label{tccone}\begin{aligned}\frac{\gamma_1-\gamma_2}{\gamma_1+\gamma_2}\pi\frac{\sqrt{2\la}}{4}\mathbf{e}_1+\sqrt{2\la}{\cal F}(\sqrt{32}\la^{-\frac12}b)+o(\sqrt\la)&=\frac{4}{3}\pi c+ o(|c|)\\ 
\end{aligned}\;\;
\hbox{ unif. for }|b|\leq M\sqrt\la\eeq
where ${\mathbf e}_1=(1,0)$ is the first canonical basis vector of  $\R^2$ and  $c=(c_1,c_2)$. 
Now, setting
$$\tilde b=\sqrt{32}\la^{-\frac12} b$$ we rewrite \eqref{tccone} as \beq\label{tccone1} \begin{aligned}\frac{\gamma_1-\gamma_2}{\gamma_1+\gamma_2}\pi\frac{\sqrt{2\la}}{4}{\mathbf e}_1+\sqrt{2\la}{\mathcal F}(\tilde b)+o(\sqrt\la)&= \frac{4}{3}\pi c+ o(|c|)
\end{aligned}\;\;\hbox{ unif. for }|\tilde b|\leq M\sqrt{32}.\eeq
The continuity of the map $b\mapsto \phi_\la=\phi_{\la,b}$ guaranteed by Proposition \ref{nonl}  implies that the term $o(\sqrt\la)$ on the left hand
side of \eqref{tccone1} is continuous too.  By assumption \eqref{pot} we deduce that  $\frac{\gamma_1-\gamma_2}{\gamma_1+\gamma_2}\in (-1,1)$. Then, according to Lemma \ref{menoo} let $\tilde b_1^*\in\R$ be such that $$J(\tilde b_1^*)= -\frac{\pi}{4}\frac{\gamma_1-\gamma_2}{\gamma_1+\gamma_2}=-\frac\pi4\frac{\gamma_1^2-\gamma_2^2}{(\gamma_1+\gamma_2)^2};$$ in particular \beq\label{rmmmm}\tilde b_1^*<0\hbox{ if } \gamma_1^2>\gamma_2^2,\quad \tilde b_1^*>0\hbox{ if } \gamma_1^2<\gamma_2^2,\quad\tilde b_1^*=0\hbox{ if }\gamma_1=\gamma_2.\eeq  Then, setting $\tilde b^*=(\tilde b^*_1,0)$, by Corollary \ref{finalcor}, ${\cal F}(\tilde b^*)=(J(\tilde b_1^*),0) =(-\frac\pi4\frac{\gamma_1-\gamma_2}{\gamma_1+\gamma_2},0) $ and, consequently,$$\frac{\gamma_1-\gamma_2}{\gamma_1+\gamma_2}\frac{\pi}{4}{\mathbf e}_1+{\mathcal F}(\tilde b^*)=0.$$  According to Corollary \ref{finalcor} we have  ${\rm det}\,{\cal F}(\tilde b^*)\neq  0$, so $0$ is a zero for the map $\tilde b\mapsto \frac{\gamma_1-\gamma_2}{\gamma_1+\gamma_2}\frac{\pi}{4}{\mathbf e}_1+{\cal F}(\tilde b)$ which is stable under uniform perturbations. Assuming that $M>0$ is sufficiently large such that $|\tilde b^*|\leq M\sqrt{32}$, then  the uniform stability gives that, if $0<\eta<1$  is sufficiently small,
then for $\la$ small enough the left hand side of \eqref{tccone1} has a zero $\tilde b_\la$ with $|\tilde b_\la-\tilde b^*|\leq \eta$ or, equivalently,
the left hand side of \eqref{tccone} has a zero $b_\la$ with $|b_\la-\sqrt{\frac{\la}{32}}\,\tilde b^*|\leq \sqrt{\frac{\la}{32}}\eta $. So for such $b_\la$  the linear system \eqref{tccone} has only the trivial solution $c_1=c_2=0$. The arbitrariness of $\eta$ implies $$b_\la-\sqrt{\frac{\la}{32}}\,\tilde b^* =o(\sqrt\la).$$ This concludes the proof of Theorem \ref{main1}.
 
\bigskip

\noindent{{\large\bf{Proof of Theorems \ref{th1} Part II: Sufficient condition.}} }Theorem \ref{main1} provides a solution to the  problem \eqref{eq1}  of the form $$v_\la=PW_\la+\phi_\la$$   for some $b=b_\la$ with $b_\la=(\frac{\la}{32})^{\frac12}\,\tilde b^*+o(\sqrt\la)$ and $\tilde b^*=(\tilde b^*_1,0)$ satisfying \eqref{rmmmm}. So  
  by \eqref{pro-exp1} we have $$v_\la= -2\log\bigg(\frac{\la}{32}+\bigg|x^2-\sqrt{\frac{\la}{32}}\,\tilde b^*\bigg|^2\bigg)+o(1).$$  
Moreover, since $|x^2-(\frac{\la}{32})^{\frac12}\,\tilde b^*|\geq c |x|^2$ with $c=\inf_{\R^2}\frac{|y-\tilde b^*|^2}{|y|^2}$, whereas  by concavity we obtain  $\log\(\frac{\la}{32}+c|x|^4\)\geq \frac12\log\(\frac{\la}{16}\)+\frac12\log (2c|x|^{4})$, consequently, 

$$ \begin{aligned}v_\la(x) +4\log|x|+\log\la&\leq -2\log\Big(\frac{\la}{32}+c|x|^4\Big)+4\log|x|+\log\la+o(1)\\ &\leq -\log\Big(\frac{\la}{16}\Big)-\log (2c|x|^{4}) +4\log|x|+\log\la+o(1)=O(1)\end{aligned}$$which implies that $v_\la$ reveals a simple blow-up profile according to \eqref{simple}.

\begin{lemma}\label{menoo}
Let
$$
 J(\xi_1)=\int_{\mathbb{R}^2}
\frac{|x+(\xi_1,0)|\,x_1}{(1+|x|^2)^3}\,dx,
\qquad \xi_1\in\mathbb{R}.
$$
Then $J\in C^2(\mathbb{R})$, $J$ is odd, $J'(\xi_1)>0$ for all $\xi_1\in\R$ and 
$$
J(\xi_1)\to \frac\pi4\hbox{ as }\xi_1\to +\infty.
$$

\end{lemma}

\begin{proof} 

Let $\xi\in\mathbb{R}^2$. By applying integration by parts we derive
$$J(\xi_1)=-\frac14\intr |x+(\xi_1,0)|\partial_{x_1}\bigg(\frac{1}{ (1+|x|^2)^2}\bigg)dx= \frac14 \intr\frac{x_1+\xi_1}{|x+\xi|}\frac{1}{(1+|x|^2)^{2}}dx=\frac14 I(\xi_1,0)$$
and the result follows directly by Lemma \ref{signlemma}

\end{proof}

As a corollary we deduce the following result.

\begin{lemma}\label{finalcor}
Let

$$
\mathcal F(\xi)
=\bigg(
\int_{\mathbb{R}^2}\frac{|x-\xi|\,x_1}{(1+|x|^2)^3}\,dx,
\int_{\mathbb{R}^2}\frac{|x-\xi|\,x_2}{(1+|x|^2)^3}\,dx
\bigg),
\qquad \xi\in\mathbb{R}^2.
$$
Then $\mathcal F\in C^1(\mathbb{R}^2)$ and the Jacobian matrix $D\mathcal F(\xi)$ is non–degenerate at every point of the form
$(\xi_1,0)$ with $\xi_1\in\R$.
\end{lemma}

\begin{proof}
The mixed derivatives vanish at points $(\xi_1,0)$:
$$\partial_{\xi_2}\mathcal F_1(\xi_1,0)=-\int_{\mathbb R^2}
\frac{x_1x_2}{|x-(\xi_1,0)|(1+|x|^2)^3}\,dx =0,$$ $$\partial_{\xi_1}\mathcal F_2(\xi_1,0)=-\int_{\mathbb R^2}
\frac{(x_1-\xi_1)x_2}{|x-(\xi_1,0)|(1+|x|^2)^3}\,dx =0,
$$ since the integrands are odd with respect to the reflection $x_2\mapsto -x_2$.
 Hence
$$
D\mathcal F(\xi_1,0)
=
\begin{pmatrix}
\partial_{\xi_1}\mathcal F_1(\xi_1,0) & 0\\
0 & \partial_{\xi_2}\mathcal F_2(\xi_1,0)
\end{pmatrix}.
$$

A direct computation gives
$$
\partial_{\xi_2}\mathcal F_2(\xi_1,0)
=-
\int_{\mathbb R^2}
\frac{x_2^2}{|x-(\xi_1,0)|(1+|x|^2)^3}\,dx <0.
$$

We now turn to $\partial_{\xi_1}\mathcal F_1(\xi_1,0)$, which can be written according to Lemma \ref{menoo} as
$$
\partial_{\xi_1}\mathcal F_1(\xi_1,0)
=J'(\xi_1)<0
$$

$$
\det D\mathcal F(\xi_1,0)
=
\partial_{\xi_1}\mathcal F_1(\xi_1,0)\,
\partial_{\xi_2}\mathcal F_2(\xi_1,0)
>0 ,
$$
and the Jacobian is non–degenerate.
\end{proof}

\end{document}